\documentclass{amsart}

\usepackage{geometry}

\usepackage[utf8]{inputenc} % Zeichensatz, ermöglicht direkte Eingabe von Umlauten

\usepackage[pdftex]{graphicx} % pdf-Grafiken 
\usepackage[ngerman,english]{babel}   % Silbentrennung nach der ndR, z.B.: Sys-tem
\usepackage{float}            % bietet Option [H] für bombenfestes Verankern von Abb
\usepackage{amsfonts}         % bietet Mengensymbole, $\mathbb{R}$
\usepackage{amssymb}
\usepackage{dsfont}
\usepackage{color}

\usepackage{bm}               % Fettdruck im Mathemodus
\usepackage{enumerate}        % verbessert Aufzählungen
\usepackage{hyperref}
\usepackage[format=hang]{caption}

\numberwithin{equation}{section}

\newcommand{\A}{\mathcal{A}}
\newcommand{\B}{\mathcal{B}}
\newcommand{\N}{\mathbb{N}}
\newcommand{\Z}{\mathbb{Z}}
\newcommand{\R}{\mathbb{R}}
\newcommand{\E}{\mathbb{E}}
\newcommand{\X}{\mathbb{X}}
\newcommand{\Prob}{\mathbb{P}}

\newtheorem{Theorem}{Theorem}
\newtheorem{Corollary}[Theorem]{Corollary}
\newtheorem{Proposition}[Theorem]{Proposition}
\newtheorem{Lemma}[Theorem]{Lemma}
\newtheorem{Remark}[Theorem]{Remark}

\newtheorem*{Assumption}{Assumption}
\theoremstyle{definition}
\newtheorem{Example}[Theorem]{Example}

\usepackage{tikz}

\usepackage{filecontents}

\immediate\write18{bibtex \jobname}

\subjclass[2010]{60H35, 65C30}
\keywords{stochastic differential equations; strong approximation; supremum error; strong asymptotic optimality; asymptotic lower error bounds; asymptotic upper error bounds; tamed Euler schemes.}

\begin{document}

\begin{abstract}
Our subject of study is strong approximation of stochastic differential equations (SDEs) with respect to the supremum error criterion, and we seek approximations that are strongly asymptotically optimal in specific classes of approximations. We hereby focus on two principal types of classes, namely, the classes of approximations that are based only on the evaluation of the initial value and on at most finitely many sequential evaluations of the driving Brownian motion on average and the classes of approximations that are based only on the evaluation of the initial value and on finitely many evaluations of the driving Brownian motion at equidistant sites. For SDEs with globally Lipschitz continuous coefficients, Müller-Gronbach [Ann. Appl. Probab. 12 (2002), no. 2, 664--690] showed that specific Euler--Maruyama schemes relating to adaptive and to equidistant time discretizations perform strongly asymptotically optimal in these classes. In the present article, we generalize these results to a significantly wider class of SDEs, such as ones with super-linearly growing coefficients. More precisely, we prove strong asymptotic optimality for specific coefficient-modified Euler--Maruyama schemes relating to adaptive and to equidistant time discretizations under rather mild assumptions on the underlying SDE. To illustrate our findings, we present two exemplary applications---namely, Euler--Maruyama schemes and tamed Euler schemes---and thereby analyze the SDE associated with the Heston--$3/2$--model originating from mathematical finance.
\end{abstract}

\title[Strongly Asymptotically Optimal Schemes for SDEs w.r.t. the Supremum Error]{Strongly Asymptotically Optimal Schemes for the \\ Strong Approximation of Stochastic Differential \\ Equations with respect to the Supremum Error}

\author{Simon Hatzesberger}
\address{Faculty of Computer Science and Mathematics, University of Passau, Innstraße 33, 94032 Passau, Germany}
%\address{Fakultät für Informatik und Mathematik, Universität Passau, Innstraße 33, 94032 Passau, Germany}
\email{simon.hatzesberger@gmail.com}
%\email{simon.hatzesberger@uni-passau.de}

\date{\today}

\maketitle

%\newpage
%%%%%%%%%%%%%%%%%%%%%%%%%%%%%%%%%%%%%%%%%%%%%%%%%%%%%%
\section{Introduction}
%\label{sec:Introduction}
%%%%%%%%%%%%%%%%%%%%%%%%%%%%%%%%%%%%%%%%%%%%%%%%%%%%%%

Let $T \in (0,\infty)$, let $d, m \in \N$, and consider a $d$-dimensional stochastic differential equation (SDE)
\begin{equation}
\label{eq:SDE_beginning}
\begin{split}
\text{d}X(t) &= \mu\big( t, X(t) \big) \, \mathrm{d}t + \sigma\big( t, X(t) \big) \, \mathrm{d}W(t), \quad t \in [0,T], \\
X(0) &= \xi,
\end{split}
\end{equation}
with drift coefficient $\mu : [0,T] \times \R^d \to \R^d$, diffusion coefficient $\sigma : [0,T] \times \R^d \to \R^{d \times m}$, $m$-dimensional Brownian motion $W$, and random initial value $\xi$ such that \eqref{eq:SDE_beginning} has a unique (strong) solution $(X(t))_{t \in [0,T]}$. In this article, we study the classes of adaptive approximations $(\X_N^{\mathrm{ad}})_{N \in \N}$ and the classes of equidistant approximations $(\X_N^{\mathrm{eq}})_{N \in \N}$. To be more specific, for each $N \in \N$, the class $\mathbb{X}_N^{\mathrm{ad}}$ consists of all approximations that are based only on the evaluation of $\xi$ and on at most $N$ sequential evaluations of $W$ on average, and the class $\mathbb{X}_N^{\mathrm{eq}}$ consists of all approximations that are based only on the evaluation of $\xi$ and on evaluations of $W$ at the equidistant sites $kT/N$, $k \in \{ 1, \ldots, N \}$. We moreover consider the error criterion
\begin{equation}
\label{eq:error_criterion}
e_q\big( \widehat{X} \big) := \bigg( \E\bigg[ \sup_{t \in [0,T]} \max_{i \in \{ 1, \ldots, d \}} \big| X_i(t) - \widehat{X}_i(t) \big|^q \bigg] \bigg)^{1/q},
\end{equation}
$q \in [1,\infty)$, which measures the $q$th mean supremum distance between the solution and a given approximation $\widehat{X} := (\widehat{X}(t))_{t \in [0,T]}$. For fixed $q \in [1,\infty)$ and $\ast \in \{ \mathrm{ad}, \mathrm{eq} \}$, the task of interest in the strong approximation problem we address is to find approximations that are strongly asymptotically optimal in the classes $(\mathbb{X}_N^\ast)_{N \in \N}$ with respect to the error $e_q$, i.e., approximations $(\widehat{X}_N)_{N \in \N}$ that satisfy $\widehat{X}_N \in \mathbb{X}_N^\ast$ for every $N \in \N$ and
\begin{equation*}
%\label{eq:asymptotic_optimality}
\lim_{N \to \infty} \frac{e_q\big( \widehat{X}_N \big)}{\inf\big\{ e_q\big( \widehat{X} \big) \; \big| \; \widehat{X} \in \mathbb{X}_N^\ast \big\}} = 1,
\end{equation*}
following the convention $0/0 := 1$ if necessary.

For the special case that the coefficients of the SDE \eqref{eq:SDE_beginning} are globally Lipschitz continuous and of at most linear growth (each with respect to the state variable), Müller-Gronbach~\cite{tmg2002} showed that specific Euler--Maruyama schemes perform strongly asymptotically optimal. In particular, the author showed strong asymptotic optimality for, on the one hand, a sequence $(\widehat{E}_N^{\mathrm{ad}})_{N \in \N}$ of piecewise-linearly interpolated Euler--Maruyama schemes on suitably constructed adaptive time discretizations in the classes $(\mathbb{X}_N^{\mathrm{ad}})_{N \in \N}$ and for, on the other hand, a sequence $(\widehat{E}_N^{\mathrm{eq}})_{N \in \N}$ of piecewise-linearly interpolated Euler--Maruyama schemes on equidistant time discretizations in the classes~$(\mathbb{X}_N^{\mathrm{eq}})_{N \in \N}$.

In the present article, we extend these results to much more general SDEs, such as ones with coefficients that may be non-globally Lipschitz continuous or super-linearly growing. To this end, we first introduce a sequence $(\widehat{X}_N^{\mathrm{ad}})_{N \in \N}$ of piecewise-linearly interpolated so-called coefficient-modified Euler--Maruyama schemes (subsequently abbreviated as modified EM schemes) on suitably constructed adaptive time discretizations as well as a sequence $(\widehat{X}_N^{\mathrm{eq}})_{N \in \N}$ of piecewise-linearly interpolated modified EM schemes on equidistant time discretizations. We then establish asymptotic upper bounds for the errors of these schemes as well as asymptotic lower bounds for the $N$th minimal errors in the classes of adaptive and of equidistant approximations. Since in both cases, the adaptive and the equidistant, the corresponding convergence rates and asymptotic constants match, it follows that the approximations $(\widehat{X}_N^{\mathrm{ad}})_{N \in \N}$ and $(\widehat{X}_N^{\mathrm{eq}})_{N \in \N}$ are strongly asymptotically optimal in the classes $(\widehat{X}_N^{\mathrm{ad}})_{N \in \N}$ and $(\widehat{X}_N^{\mathrm{eq}})_{N \in \N}$, respectively.
For these results to hold, we merely require that the modified EM schemes converge strongly to the solution of the SDE~\eqref{eq:SDE_beginning} with order~$1/2$ and that certain norms of the modified diffusion coefficients converge in $L_q$ to the respective norm of the original diffusion coefficient.
%For these results to hold, we require that the coefficient-modified Euler--Maruyama schemes converge strongly with order $1/2$ to the solution and that specific norms of the modified diffusion coefficients converge in $L_q$ to the respective norm of the original diffusion coefficient.
We finally stress that in very specific situations the adaptive time discretizations used in our modified EM schemes coincide with the ones used in Müller-Gronbach~\cite{tmg2002}.

To illustrate the scope of our results, consider the SDE associated with the Heston--3/2--model originating from mathematical finance. The scalar version of this SDE is given by
\begin{equation}
\label{eq:SDE_HES}
\begin{split}
\text{d}X(t) &= \alpha \cdot X(t) \cdot \big( \beta - |X(t)| \big) \, \mathrm{d}t + \gamma \cdot |X(t)|^{3/2} \, \mathrm{d}W(t), \quad t \in [0,T], \\
X(0) &= \xi,
\end{split}
%\tag{HES}
%\tag{HES$_{T,\xi,\alpha,\beta,\gamma}$}
\end{equation}
with parameters $d=m=1$ and $T, \alpha, \beta, \gamma, \xi \in (0,\infty)$.
%In the Theorems~\ref{th:theorem_ad} and \ref{th:theorem_eq}, we show for specific constellations of the parameters $\alpha$, $\beta$, and $\gamma$ that the asymptotics
%\begin{equation}
%\label{eq:short_ref}
%\lim_{N \to \infty} \big( N / \log(N) \big)^{1/2} \cdot \inf\Big\{ e_q\big( \widehat{X} \big) \; \Big| \; \widehat{X} \in \mathbb{X}_N^{\mathrm{ad}} \Big\} = C_q^{\mathrm{ad}}
%\end{equation}
%and
%\begin{equation}
%\label{eq:short_ref2}
%\lim_{N \to \infty} \big( N / \log(N) \big)^{1/2} \cdot \inf\Big\{ e_q\big( \widehat{X} \big) \; \Big| \; \widehat{X} \in \mathbb{X}_N^{\mathrm{ad}} \Big\} = C_q^{\mathrm{ad}}
%\end{equation}
%hold for certain values of $q \in [1,\infty)$.
Since the coefficients of the autonomous SDE~\eqref{eq:SDE_HES} are not of at most linear growth, we cannot apply the main theorems in Müller-Gronbach~\cite{tmg2002} to conclude that the particular Euler--Maruyama schemes $(\widehat{E}_N^{\mathrm{ad}})_{N \in \N}$ and $(\widehat{E}_N^{\mathrm{eq}})_{N \in \N}$ are strongly asymptotically optimal in the classes $(\mathbb{X}_N^{\mathrm{ad}})_{N \in \N}$ and $(\mathbb{X}_N^{\mathrm{eq}})_{N \in \N}$, respectively. Even worse, Theorem 1 in Hutzenthaler, Jentzen, and Kloeden~\cite{hjk2011} implies that for each $q \in [1,\infty)$ the associated errors $e_q(\widehat{E}_N^{\mathrm{ad}})$ and $e_q(\widehat{E}_N^{\mathrm{eq}})$ tend to infinity as $N$ tends to infinity. In contrast, we show in Corollary~\ref{cor:tamedEuler} that---for specific constellations of the parameters $q$, $\alpha$, $\beta$, and $\gamma$---the modified EM schemes $(\widehat{X}_N^{\mathrm{ad}})_{N \in \N}$ and $(\widehat{X}_N^{\mathrm{eq}})_{N \in \N}$ are indeed strongly asymptotically optimal in the classes $(\mathbb{X}_N^{\mathrm{ad}})_{N \in \N}$ and $(\mathbb{X}_N^{\mathrm{eq}})_{N \in \N}$, respectively. 
%We demonstrate these results by a numerical experiment which substantiates the asymptotic behavior of the errors $e_2(\widehat{X}_N^{\mathrm{ad}})$ and $e_2(\widehat{X}_N^{\mathrm{eq}})$ as $N$ tends to infinity.

We now provide a concise overview of results concerning the considered strong approximation problem. %In doing so, we focus on the subjects ‘strongly asymptotically optimal schemes for the strong approximation of (solutions of) SDEs’, ‘lower error bounds for the strong approximation of SDEs’, and ‘upper error bounds of specific methods for the strong approximation of SDEs’.
Strongly asymptotically optimal schemes with respect to the supremum error criterion~\eqref{eq:error_criterion} were first constructed by Hofmann, Müller-Gronbach, and Ritter~\cite{hofmann2000} and Müller-Gronbach~\cite{tmg2002,tmgHabil} in the case of SDEs whose coefficients are globally Lipschitz continuous and of at most linear growth. Under essentially the same assumptions, Hofmann, Müller-Gronbach, and Ritter~\cite{HMR2000a,tmg2001} and Müller-Gronbach~\cite{tmgHabil} showed respective results for the $q$th mean $L_q$ error. %Results for non-globally Lipschitz partial best of our knowledge
Lower error bounds for the strong approximation of SDEs have been extensively studied first for, again, the case of coefficients that are globally Lipschitz continuous, see, e.g., Cambanis and Hu~\cite{cambanis}, Hofmann, Müller-Gronbach, and Ritter~\cite{HMR2000a,hofmann2000,tmg2001}, and Müller-Gronbach~\cite{tmg2002,tmgHabil}. Recently, Hefter, Herzwurm, and Müller-Gronbach~\cite{tmg2017} proved lower bounds for the $N$th minimal errors in the classes of adaptive approximations that hold under rather mild assumptions on the underlying SDE; in particular, its coefficients are required to have sufficient regularity only locally, in a small neighborhood of the initial value. Whereas most of the previously mentioned results entail lower bounds with polynomial convergence rates, Jentzen, Müller-Gronbach, and Yaroslavtseva~\cite{JMY} and Yaroslavtseva~\cite{yaroslavtseva2017} constructed SDEs for which the $N$th minimal errors in certain classes converge to zero with a predefined (arbitrarily slow) convergence speed.
During the last decades, upper error bounds of specific approximations for the strong approximation of SDEs have been established mostly in the case of globally Lipschitz continuous coefficients, see, e.g., the seminal works by Maruyama~\cite{maruyama} and Milstein~\cite{milstein1995} or the book of Kloeden and Platen~\cite{kloeden1992}. The latest progress in this area is due to explicit schemes that converge strongly to the solution even if the coefficients of the considered SDE are non-globally Lipschitz continuous or super-linearly growing. For instance, we mention tamed schemes (see Hutzenthaler, Jentzen, and Kloeden~\cite{hjk2012}, Gan and Wang~\cite{ganwang2013}, Sabanis~\cite{sabanis2016}, Kumar and Sabanis~\cite{sabanistamedMilstein}, Sabanis and Zhang~\cite{sabanistamedWP}, Ngo and Luong~\cite{luong2017}), truncated schemes (see Mao~\cite{mao2015}, Guo et al.~\cite{guo}), projected schemes (see Beyn, Isaak, and Kruse~\cite{beynisaakkruse,beyn2}), and balanced schemes (see Tretyakov and Zhang~\cite{tretyakov}). For the particular case of SDEs with discontinuous coefficients, upper error bounds of Euler--Maruyama type schemes are addressed in Leobacher and Szölgyenyi~\cite{leobacher}, Ngo and Taguchi~\cite{ngo}, and Müller-Gronbach and Yaroslavtseva~\cite{mgy}. We also refer to Faure~\cite{faure1992} and Hutzenthaler, Jentzen, and Kloeden~\cite{hjk2013} for upper error bounds on piecewise-linearly interpolated Euler--Maruyama and tamed Euler schemes, respectively. 
%In Hefter and Herzwurm~\cite{hefterHerzwurm}, the authors studied a specific SDE for which the solution is the square of a one-dimensional Bessel process, and they gave upper error bounds for certain squared piecewise-constantly interpolated projected Euler schemes. 
We stress that, in contrast to our results, the asymptotic constants which can be derived from all the previously mentioned references are (up to exceptional cases) unspecified and therefore not known to be sharp.
%Again, the asymptotic constants appearing in the previously mentioned references are (up to exceptional cases) not given in an explicit form.

The remainder of this article is organized as follows.
In Section~\ref{sec:Setting}, we provide the setting and notation for the rest of this work. Moreover, we introduce conditions that will be imposed on the underlying SDE at various places in the subsequent analysis.
In Section~\ref{sec:Classes}, we formally specify what is meant by an approximation and then define the classes of adaptive and of equidistant approximations.
In Section~\ref{sec:Schemes}, we first present a continuous-time modified EM scheme. Building upon this scheme, we construct equidistant and adaptive variants in full details afterwards.
In Section~\ref{sec:MainTheorems}, we state the main results of this paper, i.e., strong asymptotic optimality of the adaptive and of the equidistant modified EM schemes in their respective classes.
In Section~\ref{sec:Applications}, we present two exemplary applications of our findings---namely, Euler--Maruyama schemes and tamed Euler schemes---and thereby conduct a numerical experiment to illustrate our results. In this context, we revisit the introductory SDE regarding the Heston--$3/2$--model.
In Section~\ref{sec:Proofs}, we carry out the proofs of our main theorems.
In Section~\ref{sec:FutureWork}, we indicate a future research direction by switching the focus from the $q$th mean supremum error to the $q$th mean $L_q$ error.
Finally, Appendix~\ref{sec:Appendix} comprises useful properties of the solution process and of the continuous-time tamed Euler schemes that will be employed in our proofs, such as moment bounds and strong convergence.

%\newpage
%%%%%%%%%%%%%%%%%%%%%%%%%%%%%%%%%%%%%%%%%%%%%%%%%%%%%%
\section{Setting, Notations, and Assumptions}
\label{sec:Setting}
%%%%%%%%%%%%%%%%%%%%%%%%%%%%%%%%%%%%%%%%%%%%%%%%%%%%%%

Throughout this article, we assume the following setting. Let $T \in (0,\infty)$, let $d, m \in \N$, let $(\Omega,\mathcal{F},\Prob)$ be a probability space with a normal filtration $(\mathcal{F}(t))_{t \in [0,T]}$, let $W : [0,T] \times \Omega \to \R^m$ be a standard $(\mathcal{F}(t))_{t \in [0,T]}$-Brownian motion on $(\Omega,\mathcal{F},\Prob)$, let $\mu : [0,T] \times \R^d \to \R^d$ be $(\mathcal{B}([0,T]) \otimes \mathcal{B}(\R^d))$-$\mathcal{B}(\R^d)$-measurable, let $\sigma : [0,T] \times \R^d \to \R^{d \times m}$ be $(\mathcal{B}([0,T]) \otimes \mathcal{B}(\R^d))$-$\mathcal{B}(\R^{d \times m})$-measurable, and let $\xi : \Omega \to \R^d$ be $\mathcal{F}(0)$-$\B(\R^d)$-measurable with finite second moment. We study the $d$-dimensional Itô stochastic differential equation
\begin{equation}
\label{eq:SDE}
\begin{split}
\text{d}X(t) &= \mu\big( t, X(t) \big) \, \mathrm{d}t + \sigma\big( t, X(t) \big) \, \mathrm{d}W(t), \quad t \in [0,T], \\
X(0) &= \xi.
\end{split}
\end{equation}
%with drift coefficient $\mu$, diffusion coefficient $\sigma$, $m$-dimensional driving Brownian motion $W$, and initial value $\xi$.

Furthermore, the following notations are used in the sequel. We denote the integer part of $z \in \R$ by $\lfloor z \rfloor := \max\{ y \in \Z \, | \, y \le z \}$, and we abbreviate the minimum of $y, z \in \R$ by $y \wedge z$. For an arbitrary set $M$, we set $\#M$ to be the cardinality of $M$ and, in case that $M \subseteq \Omega$, we define $\mathds{1}_M : \Omega \to \{ 0, 1 \}$ to be the indicator function of $M$. The Banach space of all continuous functions $f = (f_1, \ldots, f_d) : [0,T] \to \R^d$ equipped with the norm $\Vert f \Vert_\infty := \sup_{t \in [0,T]} \max_{i \in \{ 1, \ldots, d \}} |f_i(t)|$ is denoted by $(\mathcal{C}([0,T];\R^d),\Vert \cdot \Vert_\infty)$. For every $p \in (0,\infty)$ and for every random variable $Z : \Omega \to \R$, we put $\Vert Z \Vert_{L_p} := ( \E[ |Z|^p ] )^{1/p}$. For a vector $x = (x_1, \ldots, x_d) \in \R^d$, we define $x^\top$ to be the transpose of $x$ and $|x|$ to be the Euclidean norm of $x$. For a matrix $A = (A_{i,j})_{i \in \{ 1, \ldots, d \}, j \in \{ 1, \ldots, m \}} \in \R^{d \times m}$, we denote by $|A| := (\sum_{i=1}^d \sum_{j=1}^m A_{i,j}^2)^{1/2}$ the Frobenius norm of $A$ and we furthermore set $| A |_{\infty,2} := \max_{i \in \{ 1, \ldots, d \}} (\sum_{j=1}^m A_{i,j}^2)^{1/2}$. The natural exponential function and the natural logarithm function are written as $\exp : \R \to (0,\infty)$ and $\log : (0,\infty) \to \R$, respectively.

At several places in this article, we will impose additional conditions on the initial value and on the coefficients of the SDE \eqref{eq:SDE}. For $p \in [0,\infty)$ and $\varphi \in \{ \mu, \sigma \}$, we introduce the following technical assumptions here:

\begin{Assumption}[I$_p$]
\hypertarget{ass:I}{}
The initial value $\xi$ satisfies $\E\big[ |\xi|^p \big] < \infty$.
\end{Assumption}

\begin{Assumption}[locL]
\hypertarget{ass:locL}{}
The coefficients $\mu$ and $\sigma$ satisfy a local Lipschitz condition with respect to the state variable, i.e., for all $M \in \N$ there exists $C_M \in (0,\infty)$ such that for all $t \in [0,T]$ and for all $x, y \in \R^d$ with $\max\{ |x|, |y| \} \le M$ it holds that
\[ \max\Big\{ \big| \mu(t,x) - \mu(t,y) \big|, \big| \sigma(t,x) - \sigma(t,y) \big| \Big\} \le C_M \cdot | x - y |. \]
\end{Assumption}

\begin{Assumption}[H]
\hypertarget{ass:H}{}
The coefficients $\mu$ and $\sigma$ are Hölder--$1/2$--continuous with respect to the time variable with a Hölder bound that is linearly growing in the state variable, i.e., there exists \linebreak $C \in (0,\infty)$ such that for all $s,t \in [0,T]$ and for all $x \in \R^d$ it holds that
\[ \max\Big\{ \big| \mu(s,x) - \mu(t,x) \big|, \big| \sigma(s,x) - \sigma(t,x) \big| \Big\} \le C \cdot |s-t|^{1/2} \cdot \big( 1 + |x| \big). \]
\end{Assumption}

\begin{Assumption}[K$_{p}$]
\hypertarget{ass:K}{}
The coefficients $\mu$ and $\sigma$ satisfy a so-called “Khasminskii-type condition", i.e., there exists $C \in (0,\infty)$ such that for all $t \in [0,T]$ and for all $x \in \R^d$ it holds that
\[ 2 \cdot x^\top \cdot \mu(t,x) + (p-1) \cdot \big| \sigma(t,x) \big|^2 \le C \cdot \big( 1 + |x|^2 \big). \]
\end{Assumption}

\begin{Assumption}[M$_p$]
\hypertarget{ass:M}{}
The coefficients $\mu$ and $\sigma$ satisfy a so-called “monotonicity condition", i.e., there exists $C \in (0,\infty)$ such that for all $t \in [0,T]$ and for all $x, y \in \R^d$ it holds that
\[ 2 \cdot (x - y)^\top \cdot \big( \mu(t,x) - \mu(t,y) \big) + (p-1) \cdot \big| \sigma(t,x) - \sigma(t,y) \big|^2 \le C \cdot | x - y |^2. \]
\end{Assumption}

\begin{Assumption}[pG$_p^\varphi$]
\hypertarget{ass:pG}{}
The coefficient $\varphi$ grows at most polynomially in the state variable, i.e., there exists $C \in (0,\infty)$ such that for all $t \in [0,T]$ and for all $x \in \R^d$ it holds that
\[ \big| \varphi(t,x) \big| \le C \cdot \big( 1 + | x |^p \big). \]
\end{Assumption}

\begin{Assumption}[pL$_p^\varphi$]
\hypertarget{ass:pL}{}
The coefficient $\varphi$ is Lipschitz continuous with respect to the state variable with a Lipschitz bound that is polynomially growing in the state variable, i.e., there exists $C \in (0,\infty)$ such that for all $t \in [0,T]$ and for all $x, y \in \R^d$ it holds that
\[ \big| \varphi(t,x) - \varphi(t,y) \big| \le C \cdot | x - y | \, \cdot \, \big( 1 + |x|^p + |y|^p \big). \]
\end{Assumption}

It is well-known that for each $p \in [2,\infty)$ the Assumptions \hyperlink{ass:I}{\normalfont (I$_{p}$)}, \hyperlink{ass:locL}{\normalfont (locL)}, and \hyperlink{ass:K}{\normalfont (K$_{p}$)} ensure the existence of a unique solution $(X(t))_{t \in [0,T]}$ of the SDE \eqref{eq:SDE} which satisfies 
\begin{equation}
\label{eq:supEX_finite}
\sup_{t \in [0,T]} \E\Big[ \big| X(t) \big|^{p} \Big] < \infty;
\end{equation}
see, for instance, Theorem 2.4.1 in Mao~\cite{mao2008}.

\section{The Classes of Adaptive and of Equidistant Approximations}
\label{sec:Classes}
%%%%%%%%%%%%%%%%%%%%%%%%%%%%%%%%%%%%%%%%%%%%%%%%%%%%%%

In the present section, we briefly introduce the essential concepts needed to specify the classes of approximations we are interested in. To a great extent, we follow the ideas of Hefter, Herzwurm, and Müller-Gronbach \cite[Section 4]{tmg2017} and of Müller-Gronbach \cite[Section 5]{tmg2002}.

Every approximation $\widehat{X} : \, \Omega \to \mathcal{C}([0,T];\R^d)$ for the strong approximation of the solution of the SDE \eqref{eq:SDE} that is based only on the evaluation of the initial value $\xi$ and on finitely many sequential evaluations of the driving Brownian motion $W$ is determined by three sequences
\[ \psi := (\psi_k)_{k \in \N}, \quad \chi := (\chi_k)_{k \in \N}, \quad \varphi := (\varphi_k)_{k \in \N}, \]
of measurable mappings
\begin{equation*}
\begin{split}
\psi_k : \quad &\R^d \times (\R^m)^{k-1} \to (0,T], \\
\chi_k : \quad &\R^d \times (\R^m)^k \to \{ \text{STOP}, \text{GO} \}, \\
\varphi_k : \quad &\R^d \times (\R^m)^k \to \mathcal{C}([0,T];\R^d), \\
\end{split}
\end{equation*}
for $k \in \N$. Here, the sequence $\psi$ is used to obtain the sequential evaluation sites for $W$ in $(0,T]$, the sequence $\chi$ determines when to stop the evaluation of $W$, and the sequence $\varphi$ is used to get the outcome of $\widehat{X}$ once the evaluation of $W$ has stopped.
To be more specific, fix $\omega \in \Omega$ and let $x := \xi(\omega)$ and $w := W(\omega)$ be the corresponding realizations of $\xi$ and $W$, respectively. We start the evaluation of $W$ at the time point $\psi_1(x)$. After $k$ steps, we are given the data $D_k(\omega) := (x,y_1,\ldots,y_k)$ where
\[ y_1 := w\big( \psi_1(x) \big), \quad \ldots, \quad y_k := w\big( \psi_k(x,y_1,\ldots,y_{k-1}) \big), \] and we decide whether to stop or to proceed with the evaluation of $W$ according to the value of $\chi_k(D_k(\omega))$. The total number of evaluations of $W$ is given by
%$\nu : \R^d \times \mathcal{C}([0,T];\R^m) \to \N \cup \infty$ is thus given by
\[ \nu(\omega) := \min\big\{ k \in \N \; \big| \; \chi_k\big( D_k(\omega) \big) = \text{STOP} \big\}. \]
To exclude non-terminating iterations, we require $\nu < \infty$ almost surely. We eventually obtain the realization of the approximation $\widehat{X}$ by
\begin{equation}
\label{eq:scheme}
\widehat{X}(\omega) := \varphi_{\nu(\omega)}\big( D_{\nu(\omega)}(\omega) \big)
\end{equation}
in the case that $\nu(\omega) < \infty$ and arbitrarily otherwise.
%in the case that $$. 
%In order to exclude schemes that never stop the evaluation of $W$, we require $\nu < \infty$ almost surely and obtain
%\begin{equation}
%\label{eq:scheme}
%\widehat{X}(\omega) = \varphi_{\nu(\omega)}\big( D_{\nu(\omega)}(\omega) \big)
%\end{equation}
%in the case that $\nu(\omega) < \infty$. 
For technical reasons, we assume without loss of generality that for all $k, \ell \in \N$ with $k < \ell$, for all $x \in \R^d$, and for all $y \in (\R^m)^{\ell-1}$ it holds that $\psi_k(x,y_1,\ldots,y_{k-1}) \ne \psi_\ell(x,y_1,\ldots,y_{\ell-1})$.
Moreover, we denote the average number of evaluations of $W$ employed in the approximation $\widehat{X}$ by $c(\widehat{X}) := \E[ \nu ]$.

As a next step, we specify the two principal sequences of classes of approximations that are studied in this article, namely, the \textit{classes of adaptive approximations} $(\X_N^{\mathrm{ad}})_{N \in \N}$ and the \textit{classes of equidistant approximations} $(\X_N^{\mathrm{eq}})_{N \in \N}$. To this end, fix $N \in \N$ for the moment. First, the class $\X_N^{\mathrm{ad}}$ consists of all approximations that are based on the evaluation of $\xi$ and on at most $N$ sequential evaluations of $W$ on average, i.e., we define
\begin{equation*}
\begin{split}
\X_N^{\mathrm{ad}} := \Big\{ \widehat{X} : \, \Omega \to \mathcal{C}([0,T];\R^d) \; \Big| \; \widehat{X} \text{ is of the form } \eqref{eq:scheme} \text{ with } c(\widehat{X}) \le N \Big\}.
\end{split}
\end{equation*}
Second, the class $\X_N^{\mathrm{eq}}$ consists of all approximations that are based on the evaluation of $\xi$ and on evaluations of $W$ at exactly the equidistant sites $kT/N$, $k \in \{ 1, \ldots, N \}$, i.e., we define
\begin{equation*}
\begin{split}
\X_N^{\mathrm{eq}} := \Big\{ \widehat{X} : \, \Omega \to \mathcal{C}([0,T];\R^d) \; \Big| \; &\widehat{X} \text{ is of the form } \eqref{eq:scheme} \text{ with } \chi_1 = \cdots = \chi_{N-1} = \text{GO}, \\
&\chi_N = \text{STOP}, \text{ and } \psi_k = kT/N \text{ for all } k \in \{ 1, \ldots, N \} \Big\}.
\end{split}
\end{equation*}
It is easy to see that $\X_N^{\mathrm{eq}} \subseteq \X_N^{\mathrm{ad}}$ and
\begin{equation}
\label{eq:condExp}
\begin{split}
\X_N^{\mathrm{eq}} = \Big\{ u\big( \xi, W(T/&N), W(2T/N), \ldots, W(T) \big) \; \Big| \\
&u : \R^d \times (\R^m)^N \to \mathcal{C}([0,T];\R^d) \text{ is measurable} \Big\}.
\end{split}
\end{equation}

Note that the classes of adaptive and of equidistant approximations cover a large variety of important approximations. In particular, classical approximations like Euler--Maruyama type schemes corresponding to suitably chosen adaptive time discretizations (e.g., appropriate versions of the schemes presented in Fang and Giles~\cite{fang}, Kelly and Lord~\cite{kelly1}, Hofmann, Müller-Gronbach, and Ritter~\cite{HMR2000a,hofmann2000,tmg2001}, and Müller-Gronbach~\cite{tmg2002,tmgHabil}) or to equidistant time discretizations lie in the respective classes. Additionally, observe that these classes also contain even possibly non-implementable approximations like conditional expectations of the form $\E[ X \, | \, (\xi,W(T/N),W(2T/N),\ldots,W(T)) ]$, $N \in \N$, cf. the representation \eqref{eq:condExp}.

%Recall the definition \eqref{eq:error_criterion} of our error criterion. 
For $N \in \N$ and $q \in [1,\infty)$, we furthermore call $\inf\big\{ e_q\big( \widehat{X} \big) \, \big| \, \widehat{X} \in \mathbb{X}_N^{\mathrm{ad}} \big\}$ and $\inf\big\{ e_q\big( \widehat{X} \big) \, \big| \, \widehat{X} \in \mathbb{X}_N^{\mathrm{eq}} \big\}$ the \textit{$N$th minimal errors in the classes of adaptive} and \textit{of equidistant approximations}, respectively.

\begin{Remark}
The classes $(\X_N^{\mathrm{ad}})_{N \in \N}$ and $(\X_N^{\mathrm{eq}})_{N \in \N}$ introduced above are clearly not the only ones which may be studied. For example, consider the classes $(\mathbb{X}_N^{\mathrm{sn}})_{N \in \N}$ given by
\begin{equation*}
\begin{split}
\X_N^{\mathrm{sn}} := \Big\{ \widehat{X} : \, \Omega \to \mathcal{C}([0,T];\R^d) \; \Big| \; &\widehat{X} \mathrm{\;  is \; of \; the \; form \; } \eqref{eq:scheme}, \; \chi_k \mathrm{\; is \; constant \; for \; each \; } k \in \N, \\
&\mathrm{\, and \; } \nu = \min\{ k \in \N \, | \, \chi_k = \mathrm{STOP} \} \le N \Big\}
\end{split}
\end{equation*}
for $N \in \N$. The class $\X_N^{\mathrm{sn}}$ comprises all approximations that use the same number (at most $N$) of observations of $W$ for each realization and satisfies $\X_N^{\mathrm{eq}} \subseteq \X_N^{\mathrm{sn}} \subseteq \X_N^{\mathrm{ad}}$. Müller-Gronbach~\cite{tmg2002} showed strong asymptotic optimality of Euler--Maruyama schemes on specific time discretizations in these classes. Here, we solely focus on the two first-mentioned sequences of classes as these cover, in our opinion, the most interesting approximations appearing in practice.
\end{Remark}

%\newpage
%%%%%%%%%%%%%%%%%%%%%%%%%%%%%%%%%%%%%%%%%%%%%%%%%%%%%%
\section{The Equidistant and Adaptive Modified EM Schemes}
\label{sec:Schemes}
%%%%%%%%%%%%%%%%%%%%%%%%%%%%%%%%%%%%%%%%%%%%%%%%%%%%%%

In the following, we introduce two variants of coefficient-modified Euler--Maruyama type schemes (subsequently abbreviated as \emph{modified EM schemes}) that are based on equidistant and on adaptive time discretizations, respectively. The crucial ingredient for both is a continuous-time modified EM scheme which, on the one hand, is suitably close to the solution of the SDE~\eqref{eq:SDE} and which, on the other hand, possesses a simple recursive structure that will be exploited in the further analysis. The equidistant and adaptive modified EM schemes will turn out to be strongly asymptotically optimal in the classes of equidistant and of adaptive approximations, respectively.

Throughout this section, we fix functions $(\mu_N)_{N \in \N}$ and $(\sigma_N)_{N \in \N}$ such that for each $N \in \N$ the mapping $\mu_N : [0,T] \times \R^d \to \R^d$ is $(\mathcal{B}([0,T]) \otimes \mathcal{B}(\R^d))$-$\mathcal{B}(\R^d)$-measurable and the mapping $\sigma_N : [0,T] \times \R^d \to \R^{d \times m}$ is $(\mathcal{B}([0,T]) \otimes \mathcal{B}(\R^d))$-$\mathcal{B}(\R^{d \times m})$-measurable.

%Moreover, we point out that the continuous-time modified EM scheme presented here is heavily inspired by the one introduced in Sabanis~\cite{sabanis2016}. The reason we do not use the latter is that our approach is more convenient for our analysis; in particular, our scheme satisfies the desired recursion \eqref{eq:recursiveStructure} below. Nevertheless, observe that both schemes coincide in the case that the SDE \eqref{eq:SDE} is autonomous and $T=1$.

\subsection{The Continuous-time Modified EM Schemes}

Let $N \in \N$ and consider the corresponding equidistant time discretization
\begin{equation}
\label{eq:equidistant_discretization}
t_\ell^{(N)} := \ell T/N, \quad \ell \in \{ 0, \ldots, N \}.
\end{equation}
%Moreover, fix functions $\mu_N : [0,T] \times \R^d \to \R^d$ and $\sigma_N : [0,T] \times \R^d \to \R^{d \times m}$ that are assumed to be $(\mathcal{B}([0,T]) \otimes \mathcal{B}(\R^d))$-$\mathcal{B}(\R^d)$- and $(\mathcal{B}([0,T]) \otimes \mathcal{B}(\R^d))$-$\mathcal{B}(\R^{d \times m})$-measurable, respectively.
The \emph{continuous-time modified EM scheme} $\widetilde{X}_N = \widetilde{X}_N^{(\mu_N,\sigma_N)} : \, \Omega \to \mathcal{C}([0,T];\R^d)$ is defined by
\begin{equation}
\label{eq:recursiveStructure}
\begin{split}
\widetilde{X}_N(0) := &\;\xi, \\
\widetilde{X}_N(t) := &\;\widetilde{X}_N(t_\ell^{(N)}) + \mu_N\big( t_\ell^{(N)}, \widetilde{X}_N(t_\ell^{(N)}) \big) \cdot (t-t_\ell^{(N)}) \\
&+ \sigma_N\big( t_\ell^{(N)}, \widetilde{X}_N(t_\ell^{(N)}) \big) \cdot \big( W(t) - W(t_\ell^{(N)}) \big),
\end{split}
\end{equation}
for all $\ell \in \{ 0, \ldots, N-1 \}$ and for all $t \in (t_\ell^{(N)},t_{\ell+1}^{(N)}]$.

Note that this stochastic process possesses the following Itô representation: almost surely we have
\begin{equation}
\label{eq:ItoStructure}
\begin{split}
\widetilde{X}_N(t) = &\;\xi + \int_0^t \mu_N\big( \lfloor sN/T \rfloor \cdot T/N, \widetilde{X}_N(\lfloor sN/T \rfloor \cdot T/N) \big) \, \mathrm{d}s \\
&+ \int_0^t \sigma_N\big( \lfloor sN/T \rfloor \cdot T/N, \widetilde{X}_N( \lfloor sN/T \rfloor \cdot T/N ) \big) \, \mathrm{d}W(s)
\end{split}
\end{equation}
for all $t \in [0,T]$.

Since entire trajectories of $W$ are used in the above construction of~$\widetilde{X}_N$, this scheme is not an approximation in the sense of Section~\ref{sec:Classes}, and we thus find $\widetilde{X}_N \not\in \X_M^{\mathrm{ad}}$ for every $M \in \N$. % for all $N \in \N$ and $r \in [0,\infty)$.

%Some assumptions for $q \in [1,\infty)$:
%
%\begin{Assumption}[C$_{q}$]
%\hypertarget{ass:C}{}
%The continuous-time meta Euler scheme converges strongly to the solution with order (at least) $1/2$, i.e., there exists $C \in (0,\infty)$ sucht that for all $N \in \N$ it holds that
%\[ \Big\Vert \big\Vert X - \widetilde{X}_N \big\Vert_\infty \Big\Vert_{L_q} \le C \cdot N^{-1/2}. \]
%\end{Assumption}
%
%\begin{Assumption}[X$_{q}$]
%\hypertarget{ass:X}{}
%The continuous-time meta Euler scheme satisfies
%\[ \sup_{N \in \N} \E\bigg[ \sup_{t \in [0,T]} \big| \widetilde{X}_N(t) \big|^q \bigg] < \infty. \]
%\end{Assumption}
%
%\begin{Assumption}[Pr]
%\hypertarget{ass:Pr}{}
%It holds that
%\[ \sup_{t \in [0,T]} \Big| \sigma\big( t, X(t) \big) - \sigma_N\big( t, \widetilde{X}_N(t) \big) \Big| \quad \xrightarrow[N \to \infty]{\Prob} \quad 0. \]
%\end{Assumption}

\subsection{The Equidistant Modified EM Schemes}

Next, based on the continuous-time modified EM schemes, we construct approximations which use not whole paths of the driving Brownian motion but evaluate $W$ only at equidistant sites.

As before, let $N \in \N$ and consider the equidistant time discretization \eqref{eq:equidistant_discretization}. The \emph{equidistant modified EM scheme} $\widehat{X}_N^{\mathrm{eq}} : \, \Omega \to \mathcal{C}([0,T];\R^d)$ is defined by
\begin{equation*}
\begin{split}
\widehat{X}_N^{\mathrm{eq}}(t_\ell^{(N)}) := \widetilde{X}_N(t_\ell^{(N)})
\end{split}
\end{equation*}
for all $\ell \in \{ 0, \ldots, N \}$ and linearly interpolated between these time points.

By suitably choosing sequences $\psi$, $\chi$, and $\varphi$ as per Section \ref{sec:Classes}, we obtain $\widehat{X}_N^{\mathrm{eq}} \in \X_N^{\mathrm{eq}}$. Clearly, the total number of evaluations of $W$ employed in the approximation $\widehat{X}_N^{\mathrm{eq}}$ is given by $N$ for each realization. % for all $N \in \N$ and $r \in [0,\infty)$.

\subsection{The Adaptive Modified EM Schemes}
%\label{subsec:AdaptiveModifiedEM}

The following construction of the adaptive modified EM schemes is heavily inspired by the construction of the adaptive Euler schemes presented in Müller-Gronbach \cite[Subsection 3.1]{tmg2002}.

Note that, under suitable regularity assumptions on the coefficients of the SDE \eqref{eq:SDE}, its solution $(X(t))_{t \in [0,T]}$ satisfies
\[ \E\Big[ \big| X_i(t+\delta) - X_i(t) \big|^2 \, \Big| \, X(t) \Big] = \sum_{j=1}^m \big|\sigma_{i,j}\big(t,X(t) \big)\big|^2 \cdot \delta + o(\delta) \]
for all $i \in \{ 1, \ldots, d \}$ and for all $t \in [0,T]$. Hence, the paths of each component $X_i$ are, in the root mean square sense and conditioned on $X(t)$, locally Hölder--$1/2$--continuous with Hölder constant $(\sum_{j=1}^m |\sigma_{i,j}(t,X(t))|^2)^{1/2}$, and the maximum over $i \in \{ 1, \ldots, d \}$ of all these constants is given by $\vert \sigma(t,X(t)) \vert_{\infty,2}$. For this reason, it is more beneficial to evaluate $W$ more often in regions where the value of $\vert \sigma(t,X(t)) \vert_{\infty,2}$ is large and vice versa.

Motivated by this idea, we construct our adaptive modified EM scheme in two steps. First, we use equidistant time steps to roughly approximate the solution and thereby obtain estimates for the conditional Hölder constants at these sites. Second, we refine our approximation by taking into account the local smoothness of the solution. For this purpose, we distribute additional evaluation sites between those equidistant time points for which the corresponding maximum of the estimated Hölder constants is large in proportion to the other time points.
%Motivated by this idea, we construct our adaptive modified EM scheme in two steps. First, we use equidistant time steps to roughly approximate the solution and thereby obtain estimates for the conditional Hölder constants at these sites. Second, we refine our approximation between those equidistant time points for which the corresponding maximum of the estimated Hölder constant is large in proportion to the totality of the maxima of the estimated Hölder constants.
%Motivated by this idea, we construct our adaptive tamed Euler scheme in two steps. First, we use equidistant time steps to roughly approximate the solution and thereby obtain estimates for the conditional Hölder constants at these sites. Second, we refine our approximation between those equidistant time points for which the corresponding estimated Hölder constant is large in proportion to the totality of the estimated Hölder constants.

Let $q \in [1,\infty)$, let $r \in [0,\infty)$, and let $(k_N)_{N \in \N}$ be a sequence of natural numbers satisfying
\begin{equation}
\label{eq:limits_kN}
\lim_{N \to \infty} \frac{k_N}{N} = 0 = \lim_{N \to \infty} \frac{N}{k_N \cdot \log(N)}.
\end{equation}
%Each value $T/k_N$, $N \in \N$, will serve as a step size for the rough approximation via an equidistant discretization.
Fix $N \in \N$ and put
\begin{equation}
\label{eq:tildeA_kN}
\A_{k_N} := \Bigg( \frac{T}{k_N} \cdot \sum_{\ell=0}^{k_N-1} \big\vert \sigma_{k_N}\big( t_\ell^{(k_N)}, \widetilde{X}_{k_N}(t_\ell^{(k_N)}) \big) \big\vert_{\infty,2}^2 \Bigg)^{1/2}.
\end{equation}
For each $\ell \in \{ 0, \ldots, k_N-1 \}$, we consider the random discretization
\begin{equation}
\label{eq:adaptive_discretization}
t_\ell^{(k_N)} = \tau_{\ell,0}^{(k_N)} < \tau_{\ell,1}^{(k_N)} < \ldots < \tau_{\ell,\eta_\ell+1}^{(k_N)} = t_{\ell+1}^{(k_N)}
\end{equation}
of $[t_\ell^{(k_N)},t_{\ell+1}^{(k_N)}]$ where
\begin{equation}
\label{eq:eta_l}
\eta_\ell := \mathds{1}_{\{ \A_{k_N} > 0 \}} \cdot \left\lfloor N \cdot \A_{k_N}^{2q/(q+2)} \cdot \dfrac{\big\vert \sigma_{k_N}\big( t_\ell^{(k_N)}, \widetilde{X}_{k_N}(t_\ell^{(k_N)}) \big) \big\vert_{\infty,2}^2}{\displaystyle \sum\limits_{\iota=0}^{k_N-1} \big\vert \sigma_{k_N}\big( t_\iota^{(k_N)}, \widetilde{X}_{k_N}(t_\iota^{(k_N)}) \big) \big\vert_{\infty,2}^2} \right\rfloor
\end{equation}
and
\[ \tau_{\ell,\kappa}^{(k_N)} := t_\ell^{(k_N)} + \frac{T}{k_N} \cdot \frac{\kappa}{\eta_\ell + 1} \]
for all $\kappa \in \{ 0, \ldots, \eta_\ell+1 \}$.
%Due to the construction of the numbers $\eta_\ell$, $\ell \in \{ 0, \ldots, k_N-1 \}$, we indeed distribute more evaluation sites to regions where we want to refine our approximation according to the motivation above.
The \emph{adaptive modified EM scheme} $\widehat{X}_{N,q}^{\mathrm{ad}} : \, \Omega \to \mathcal{C}([0,T];\R^d)$ is defined by
\begin{equation*}
\begin{split}
\widehat{X}_{N,q}^{\mathrm{ad}}(\tau_{\ell,\kappa+1}^{(k_N)}) := &\,\widehat{X}_{N,q}^{\mathrm{ad}}(\tau_{\ell,\kappa}^{(k_N)}) + \mu_{k_N}\big( t_\ell^{(k_N)}, \widetilde{X}_{k_N}(t_\ell^{(k_N)}) \big) \cdot (\tau_{\ell,\kappa+1}^{(k_N)} - \tau_{\ell,\kappa}^{(k_N)}) \\
&+ \sigma_{k_N}\big( t_\ell^{(k_N)}, \widetilde{X}_{k_N}(t_\ell^{(k_N)}) \big) \cdot \big( W(\tau_{\ell,\kappa+1}^{(k_N)}) - W(\tau_{\ell,\kappa}^{(k_N)}) \big)
\end{split}
\end{equation*}
for all $\ell \in \{ 0, \ldots, k_N-1 \}$ and for all $\kappa \in \{ 0, \ldots, \eta_\ell+1 \}$, and linearly interpolated between all these time points.

By suitably choosing sequences $\psi$, $\chi$, and $\varphi$ as per Section \ref{sec:Classes}, we obtain $\widehat{X}_{N,q}^{\mathrm{ad}} \in \X_{\lceil c(\widehat{X}_{N,q}^{\mathrm{ad}}) \rceil}^{\mathrm{ad}}$ provided that $c(\widehat{X}_{N,q}^{\mathrm{ad}}) < \infty$. 
Define $\nu_{N,q}^{\mathrm{ad}}$ to be the (random) number of evaluations of $W$ employed in the approximation $\widehat{X}_{N,q}^{\mathrm{ad}}$. Then we clearly have
\begin{equation}
\label{eq:inequality_totalNumber_nachOben}
\nu_{N,q}^{\mathrm{ad}} = k_N + \sum_{\ell=0}^{k_N-1} \eta_\ell \le k_N + N \cdot \A_{k_N}^{2q/(q+2)}
\end{equation}
and
\begin{equation}
\label{eq:inequality_totalNumber_nachUnten}
\nu_{N,q}^{\mathrm{ad}} \ge \max\Big\{ k_N, \, k_N + \mathds{1}_{\{ \A_{k_N} > 0 \}} \cdot \big( N \cdot \A_{k_N}^{2q/(q+2)} - k_N \big) \Big\}.
\end{equation}

%\newpage
%%%%%%%%%%%%%%%%%%%%%%%%%%%%%%%%%%%%%%%%%%%%%%%%%%%%%%
\section{Main Results}
\label{sec:MainTheorems}
%%%%%%%%%%%%%%%%%%%%%%%%%%%%%%%%%%%%%%%%%%%%%%%%%%%%%%

The following theorems entirely specify the asymptotics of the $N$th minimal errors in the classes of adaptive and of equidistant approximations as well as the asymptotics of the errors of the adaptive and of the equidistant modified EM schemes. As a consequence, we will conclude strong asymptotic optimality of these approximations in their respective classes. The proofs of all theorems are postponed to Section~\ref{sec:Proofs}. %Finally, we compare our results to those in Müller-Gronbach~\cite{tmg2002} at the end of this section.

In the case that the SDE \eqref{eq:SDE} has a unique solution $(X(t))_{t \in [0,T]}$, we put
\begin{equation*}
\begin{split}
C_q^{\mathrm{ad}} &:= 2^{-1/2} \cdot \bigg\Vert \bigg( \int_0^T \big\vert \sigma\big( t, X(t) \big) \big\vert_{\infty,2}^2 \, \mathrm{d}t \bigg)^{1/2} \bigg\Vert_{L_{2q/(q+2)}}, \\
C_q^{\mathrm{eq}} &:= (T/2)^{1/2} \cdot \bigg\Vert \sup_{t \in [0,T]} \big\vert \sigma\big( t, X(t) \big) \big\vert_{\infty,2} \bigg\Vert_{L_q},
\end{split}
\end{equation*}
for $q \in [1,\infty)$. The quantities $C_q^{\mathrm{ad}}$ and $C_q^{\mathrm{eq}}$ will turn out to be the sharp asymptotic constants for the $N$th minimal errors in the classes $(\X_N^{\mathrm{ad}})_{N \in \N}$ and $(\X_N^{\mathrm{eq}})_{N \in \N}$, respectively. It is easy to see that $C_q^{\mathrm{ad}} \le C_q^{\mathrm{eq}}$ holds for all $q \in [1,\infty)$. The succeeding remarks provide sufficient conditions for the finiteness of these two constants as well as sufficient and necessary conditions for them being identical (to zero).

\begin{Remark}
Let the Assumptions \hyperlink{ass:I}{\normalfont (I$_{p}$)}, \hyperlink{ass:locL}{\normalfont (locL)}, \hyperlink{ass:K}{\normalfont (K$_{p}$)}, and \hyperlink{ass:pG}{\normalfont (pG$_r^{\sigma}$)} be satisfied for some $p \in [2,\infty)$ and $r \in [1,\infty)$ with $p \ge \max\{ 2r, 3r-2 \}$. Then Proposition \ref{prop:EsupX} in Appendix \ref{sec:Appendix} implies $C_{(p-2r+2)/r}^{\mathrm{eq}} < \infty$.
\end{Remark}

\begin{Remark}
Fix $q \in [1,\infty)$ such that the SDE \eqref{eq:SDE} has a unique solution $(X(t))_{t \in [0,T]}$ which satisfies $C_q^{\mathrm{eq}} < \infty$. Then we have $C_q^{\mathrm{ad}} = C_q^{\mathrm{eq}}$ if and only if almost surely it holds that the mapping $[0,T] \to \R, \; t \mapsto \vert \sigma(t,X(t)) \vert_{\infty,2},$ is constant, and we have $C_q^{\mathrm{ad}} = C_q^{\mathrm{eq}} = 0$ if and only if almost surely it holds that the mapping $[0,T] \to \R, \; t \mapsto \vert \sigma(t,X(t)) \vert_{\infty,2},$ is constantly zero; cf. Remark 1 in Müller-Gronbach~\cite{tmg2002}.
\end{Remark}

First, we specify the asymptotics of the $N$th minimal errors in the classes of adaptive approximations as well as the asymptotics of the errors of the adaptive modified EM schemes. More precisely, we not only state the convergence rates but also give the asymptotic constants. Since both the convergence rates and the asymptotic constants match, we obtain strong asymptotic optimality of the adaptive modified EM schemes in the classes of adaptive approximations.

\begin{Theorem}
\label{th:theorem_ad}
Fix $q \in [1,\infty)$ such that the SDE \eqref{eq:SDE} has a unique solution $(X(t))_{t \in [0,T]}$ which satisfies $C_q^{\mathrm{ad}} < \infty$, and fix measurable functions $(\mu_N)_{N \in \N}$ and $(\sigma_N)_{N \in \N}$ as per Section~\ref{sec:Schemes}. Moreover, assume that there exists $C \in (0,\infty)$ such that for all $N \in \N$ it holds that
\begin{equation}
\label{eq:theorem_ad_strong_convergence}
\Big\Vert \big\Vert X - \widetilde{X}_N \big\Vert_\infty \Big\Vert_{L_q} \le C \cdot N^{-1/2},
\end{equation}
and assume that
\begin{equation}
\label{eq:theorem_ad_convergence_in_Lq}
\bigg( \frac{T}{N} \cdot \sum_{\ell=0}^{N-1} \big\vert \sigma_N\big(t_\ell^{(N)}, \widetilde{X}_N(t_\ell^{(N)}) \big) \big\vert_{\infty,2}^2 \bigg)^{1/2} \quad \xrightarrow[N \to \infty]{L_q} \quad \bigg( \int_0^T \big\vert \sigma\big( t, X(t) \big) \big\vert_{\infty,2}^2 \, \mathrm{d}t \bigg)^{1/2}.
\end{equation}
Then it holds that
\begin{equation}
\label{eq:theorem_minErrors_ad}
\lim_{N \to \infty} \big( N/\log(N) \big)^{1/2} \cdot \inf\Big\{ e_q\big( \widehat{X} \big) \; \Big| \; \widehat{X} \in \mathbb{X}_N^{\mathrm{ad}} \Big\} = C_q^{\mathrm{ad}}
\end{equation}
%as well as
%\[ c\big(\widehat{X}_{N,q}^{\mathrm{ad}}\big) < \infty \]
%for each $N \in \N$
and
\begin{equation}
\label{eq:theorem_tamedSchemes_ad}
\lim_{N \to \infty} \Big( c\big(\widehat{X}_{N,q}^{\mathrm{ad}}\big) / \log\big( c\big(\widehat{X}_{N,q}^{\mathrm{ad}}\big) \big) \Big)^{1/2} \cdot e_q\big( \widehat{X}_{N,q}^{\mathrm{ad}} \big) = C_q^{\mathrm{ad}}.
\end{equation}
In the case $C_q^{\mathrm{ad}} > 0$, we conclude that the approximations $(\widehat{X}_{N,q}^{\mathrm{ad}})_{N \in \N}$ are strongly asymptotically optimal in the classes $(\X_{\lceil c(\widehat{X}_{N,q}^{\mathrm{ad}}) \rceil}^{\mathrm{ad}})_{N \in \N}$.
%In the case $C_q^{\mathrm{ad}} > 0$, we immediately conclude that the approximations $(\widehat{X}_{N,q}^{\mathrm{ad}})_{N \in \N}$ are strongly asymptotically optimal in the classes $(\X_{\lceil c(\widehat{X}_{N,q}^{\mathrm{ad}}) \rceil}^{\mathrm{ad}})_{N \in \N}$.
\end{Theorem}

\begin{proof}
This result is an immediate consequence of the Lemmas~\ref{lemma:liminf_ad} and \ref{lemma:limsup_ad} given in Section~\ref{sec:Proofs}.
\end{proof}

Next, we specify the asymptotics of the $N$th minimal errors in the classes of equidistant approximations as well as the asymptotics of the errors of the equidistant modified EM schemes. More precisely, we not only state the convergence rates but also give the asymptotic constants. Since both the convergence rates and the asymptotic constants match, we obtain strong asymptotic optimality of the equidistant modified EM schemes in the classes of equidistant approximations.

\begin{Theorem}
\label{th:theorem_eq}
Fix $q \in [1,\infty)$ such that the SDE \eqref{eq:SDE} has a unique solution $(X(t))_{t \in [0,T]}$ which satisfies $C_q^{\mathrm{eq}} < \infty$, and fix measurable functions $(\mu_N)_{N \in \N}$ and $(\sigma_N)_{N \in \N}$ as per Section~\ref{sec:Schemes}. Moreover, assume that there exists $C \in (0,\infty)$ such that for all $N \in \N$ it holds that
\begin{equation}
\label{eq:theorem_eq_strong_convergence}
\Big\Vert \big\Vert X - \widetilde{X}_N \big\Vert_\infty \Big\Vert_{L_q} \le C \cdot N^{-1/2},
\end{equation}
and assume that
\begin{equation}
\label{eq:theorem_eq_convergence_in_Lq}
\max_{\ell \in \{ 0, \ldots, N-1 \}} \big\vert \sigma_N\big(t_\ell^{(N)}, \widetilde{X}_N(t_\ell^{(N)}) \big) \big\vert_{\infty,2} \quad \xrightarrow[N \to \infty]{L_q} \quad \sup_{t \in [0,T]} \big\vert \sigma\big( t, X(t) \big) \big\vert_{\infty,2}. 
\end{equation}
Then it holds that
\begin{equation}
\label{eq:theorem_minErrors_eq}
\lim_{N \to \infty} \big( N/\log(N) \big)^{1/2} \cdot \inf\Big\{ e_q\big( \widehat{X} \big) \; \Big| \; \widehat{X} \in \mathbb{X}_N^{\mathrm{eq}} \Big\} = C_q^{\mathrm{eq}}
\end{equation}
and
\begin{equation}
\label{eq:theorem_tamedSchemes_eq}
\lim_{N \to \infty} \big( N/\log(N) \big)^{1/2} \cdot e_q\big( \widehat{X}_N^{\mathrm{eq}} \big) = C_q^{\mathrm{eq}}.
\end{equation}
In the case $C_q^{\mathrm{eq}} > 0$, we conclude that the approximations $(\widehat{X}_N^{\mathrm{eq}})_{N \in \N}$ are strongly asymptotically optimal in the classes $(\X_N^{\mathrm{eq}})_{N \in \N}$.
%In the case $C_q^{\mathrm{eq}} > 0$, we immediately conclude that the approximations $(\widehat{X}_N^{\mathrm{eq}})_{N \in \N}$ are strongly asymptotically optimal in the classes $(\X_N^{\mathrm{eq}})_{N \in \N}$.
\end{Theorem}

\begin{proof}
This result is an immediate consequence of the Lemmas \ref{lemma:liminf_eq} and \ref{lemma:limsup_eq} given in Section~\ref{sec:Proofs}.
\end{proof}

\begin{Remark}
We briefly comment on the assumptions of the theorems above. First, we consider the equidistant case. Clearly, existence of a unique solution and finiteness of the asymptotic constant~$C_q^{\mathrm{eq}}$ ensure well-posedness of the considered problem. A crucial step in our analysis of $e_q(\widehat{X}_{N}^{\mathrm{eq}})$, $N \in \N$, will be to split this error into two parts: the distance between $X$ and $\widetilde{X}_N$ and the distance between $\widetilde{X}_N$ and $\widehat{X}_{N}^{\mathrm{eq}}$, see \eqref{eq:test1}. The strong convergence order $1/2$ of the continuous-time EM schemes as given by~\eqref{eq:theorem_eq_strong_convergence} will imply that the former distances become asymptotically negligible in comparison to the latter ones. And exactly these latter distances will turn out to be determined by specific weighted Brownian bridges; the $L_q$ convergence~\eqref{eq:theorem_eq_convergence_in_Lq} will then, in particular, establish a central asymptotic relation between these weights and the asymptotic constant. Similar arguments also apply for the adaptive case.
\end{Remark}

\section{Applications}
\label{sec:Applications}
%%%%%%%%%%%%%%%%%%%%%%%%%%%%%%%%%%%%%%%%%%%%%%%%%%%%%%

We now present two exemplary applications of our main theorems. First, we show strong asymptotic optimality of the classical Euler--Maruyama schemes relating to specific adaptive and equidistant discretizations in the setting of SDEs with globally Lipschitz continuous coefficients. Afterwards, we derive that specific adaptive and equidistant tamed Euler schemes are strongly asymptotically optimal for certain SDEs whose coefficients may grow polynomially.

\subsection{Euler--Maruyama Schemes}
%\label{subsec:Ex_Euler_schemes}
%%%%%%%%%%%%%%%%%%%%%%%%%%%%%%%%%%%%%%%%%%%%%%%%%%%%%%

Let the functions $(\mu_N)_{N \in \N}$ and $(\sigma_N)_{N \in \N}$ be given by
\begin{equation}
\label{eq:muN_sigmaN_Euler}
\mu_N = \mu, \quad \sigma_N = \sigma, \quad N \in \N.
\end{equation}
Then the modified EM schemes determined by these functions coincide with the corresponding classical Euler--Maruyama schemes. In the subsequent corollary, we prove strong asymptotic optimality of their adaptive and equidistant variants in the classes of adaptive and of equidistant approximations, respectively.

\begin{Corollary}
\label{cor:Euler}
Fix $q \in [1,\infty)$ and let the Assumptions~\hyperlink{ass:I}{\normalfont (I$_{\max\{q,2\}}$)}, \hyperlink{ass:H}{\normalfont (H)}, \hyperlink{ass:pL}{\normalfont (pL$_{0}^{\mu}$)}, and \hyperlink{ass:pL}{\normalfont (pL$_{0}^{\sigma}$)} be satisfied. Moreover, let the functions $(\mu_N)_{N \in \N}$ and $(\sigma_N)_{N \in \N}$ be given by \eqref{eq:muN_sigmaN_Euler}. Then the asymptotics~\eqref{eq:theorem_minErrors_ad}, \eqref{eq:theorem_tamedSchemes_ad}, \eqref{eq:theorem_minErrors_eq}, and \eqref{eq:theorem_tamedSchemes_eq} hold true. If we additionally have $C_q^{\mathrm{ad}} > 0$, then the Euler--Maruyama schemes $(\widehat{X}_{N,q}^{\mathrm{ad}})_{N \in \N}$ and $(\widehat{X}_{N}^{\mathrm{eq}})_{N \in \N}$ are strongly asymptotically optimal in the classes $(\X_{\lceil c(\widehat{X}_{N,q}^{\mathrm{ad}}) \rceil}^{\mathrm{ad}})_{N \in \N}$ and $(\X_N^{\mathrm{eq}})_{N \in \N}$, respectively.
\end{Corollary}

\begin{proof}
First of all, the Assumptions~\hyperlink{ass:H}{\normalfont (H)}, \hyperlink{ass:pL}{\normalfont (pL$_{0}^{\mu}$)}, and \hyperlink{ass:pL}{\normalfont (pL$_{0}^{\sigma}$)} imply the linear growth conditions \hyperlink{ass:pG}{\normalfont (pG$_{1}^{\mu}$)} and \hyperlink{ass:pG}{\normalfont (pG$_{1}^{\sigma}$)}. Hence, by using Proposition~\ref{prop:EsupX} from Appendix~\ref{sec:Appendix}, we conclude
\begin{equation}
\label{eq:appl_help1}
\E\bigg[ \sup_{t \in [0,T]} \big| X(t) \big|^{\max\{q,2\}} \bigg] < \infty
\end{equation}
and thereby prove $C_q^{\mathrm{eq}} < \infty$. 

Furthermore, it is well-known that in the considered setting the continuous-time Euler--Maruyama schemes convergence strongly with order $1/2$, see, for instance, Proposition~14 in Faure~\cite{faure1992}; we consequently have \eqref{eq:theorem_ad_strong_convergence} and \eqref{eq:theorem_eq_strong_convergence}. 

Finally, the $L_q$ convergences \eqref{eq:theorem_ad_convergence_in_Lq} and \eqref{eq:theorem_eq_convergence_in_Lq} essentially follow from \hyperlink{ass:H}{\normalfont (H)}, \hyperlink{ass:pL}{\normalfont (pL$_{0}^{\sigma}$)}, \hyperlink{ass:pL}{\normalfont (pG$_{1}^{\sigma}$)}, and \eqref{eq:appl_help1}.

Applying Theorems~\ref{th:theorem_ad} and \ref{th:theorem_eq} finishes the proof of this corollary.
\end{proof}

\begin{Remark}
The results on the Euler--Maruyama schemes presented in Corollary~\ref{cor:Euler} are well-known and were first (directly) proved by Müller-Gronbach~\cite{tmg2002} for the particular case $T=1$.
\end{Remark}

\subsection{Tamed Euler Schemes}
%\label{subsec:Ex_TamedEuler}
%%%%%%%%%%%%%%%%%%%%%%%%%%%%%%%%%%%%%%%%%%%%%%%%%%%%%%

%Let $N \in \N$ and $r \in [0,\infty)$, and consider 
%\[ \mu_N : \quad [0,T] \times \R^d \to \R^d, \quad (t,x) \mapsto \frac{\mu(t,x)}{1+(T/N)^{1/2} \cdot |x|^r}, \]
%and 
%\[ \sigma_N : \quad [0,T] \times \R^d \to \R^{d \times m}, \quad (t,x) \mapsto \frac{\sigma(t,x)}{1+(T/N)^{1/2} \cdot |x|^r}. \]
Let $r \in [0,\infty)$ and let the functions $(\mu_N)_{N \in \N}$ and $(\sigma_N)_{N \in \N}$ be given by
\begin{equation}
\label{eq:muN_sigmaN_tamedEuler}
\begin{split}
\mu_N = \mu_N^{(r)} : \quad [0,T] \times \R^d \to \R^d, \quad &(t,x) \mapsto \frac{\mu(t,x)}{1+(T/N)^{1/2} \cdot |x|^r}, \\ 
\sigma_N = \sigma_N^{(r)} : \quad [0,T] \times \R^d \to \R^{d \times m}, \quad &(t,x) \mapsto \frac{\sigma(t,x)}{1+(T/N)^{1/2} \cdot |x|^r}, \quad N \in \N.
\end{split}
\end{equation}
Then the modified EM schemes determined by these functions constitute so-called tamed Euler schemes.
In the subsequent corollary, we prove strong asymptotic optimality of their adaptive and equidistant variants in the classes of adaptive and of equidistant approximations, respectively.

We stress that the type of continuous-time tamed Euler schemes considered here is heavily inspired by the one introduced in Sabanis~\cite{sabanis2016}. The reason we do not use the latter is that our approach is more convenient for our analysis; in particular, our schemes satisfy the desired recursion~\eqref{eq:recursiveStructure}. Nevertheless, observe that both types of tamed Euler schemes coincide in the case that the SDE~\eqref{eq:SDE} is autonomous and $T=1$.

\begin{Corollary}
\label{cor:tamedEuler}
Fix $q \in [1,\infty)$ and let the Assumptions~\hyperlink{ass:I}{\normalfont (I$_p$)}, \hyperlink{ass:H}{\normalfont (H)}, \hyperlink{ass:K}{\normalfont (K$_p$)}, \hyperlink{ass:M}{\normalfont (M$_a$)}, and \hyperlink{ass:pL}{\normalfont (pL$_r^\mu$)} be satisfied for some $p,a \in [2,\infty)$ and $r \in [0,\infty)$ such that $p \ge 4r+2$ and $q < \min\{ a, p/(2r+1) \}$. Moreover, let the functions $(\mu_N)_{N \in \N}$ and $(\sigma_N)_{N \in \N}$ be given by \eqref{eq:muN_sigmaN_tamedEuler}. Then the asymptotics~\eqref{eq:theorem_minErrors_ad}, \eqref{eq:theorem_tamedSchemes_ad}, \eqref{eq:theorem_minErrors_eq}, and \eqref{eq:theorem_tamedSchemes_eq} hold true. If we additionally have $C_q^{\mathrm{ad}} > 0$, then the tamed Euler schemes $(\widehat{X}_{N,q}^{\mathrm{ad}})_{N \in \N}$ and $(\widehat{X}_{N}^{\mathrm{eq}})_{N \in \N}$ are strongly asymptotically optimal in the classes $(\X_{\lceil c(\widehat{X}_{N,q}^{\mathrm{ad}}) \rceil}^{\mathrm{ad}})_{N \in \N}$ and $(\X_N^{\mathrm{eq}})_{N \in \N}$, respectively.
\end{Corollary}

\begin{proof}
Observe first that in this setting the conditions~\hyperlink{ass:locL}{\normalfont (locL)}, \hyperlink{ass:pG}{\normalfont (pG$_{r+1}^{\mu}$)}, \hyperlink{ass:pL}{\normalfont (pL$_{r/2}^{\sigma}$)}, and \hyperlink{ass:pG}{\normalfont (pG$_{(r+2)/2}^{\sigma}$)} are also satisfied. Hence, by using Proposition~\ref{prop:EsupX} from Appendix~\ref{sec:Appendix}, we conclude
\begin{equation}
\label{eq:appl_help2}
\E\bigg[ \sup_{t \in [0,T]} \big| X(t) \big|^{p-r} \bigg] < \infty
\end{equation}
and thereby prove $C_q^{\mathrm{eq}} < \infty$.

The strong convergence properties \eqref{eq:theorem_ad_strong_convergence} and \eqref{eq:theorem_eq_strong_convergence} immediately follow from Proposition~\ref{prop:dist_X_tildeX} given in Appendix~\ref{sec:Appendix}.

Our proof for the $L_q$ convergences \eqref{eq:theorem_ad_convergence_in_Lq} and \eqref{eq:theorem_eq_convergence_in_Lq} is slightly more demanding and will be carried out in several steps.
As a first step, we show
\begin{equation}
\label{eq:ch3_inProbability}
\sup_{t \in [0,T]} \Big| \sigma\big( t, X(t) \big) - \sigma_N\big( t, \widetilde{X}_N(t) \big) \Big| \quad \xrightarrow[N \to \infty]{\Prob} \quad 0.
\end{equation}
Due to the triangle inequality, it suffices to prove
\begin{equation}
\label{eq:ch3_inProb1}
\sup_{t \in [0,T]} \Big| \sigma\big( t, X(t) \big) - \sigma\big( t, \widetilde{X}_{N}(t) \big) \Big| \quad \xrightarrow[N \to \infty]{\Prob} \quad 0
\end{equation}
and
\begin{equation}
\label{eq:ch3_inProb2}
\sup_{t \in [0,T]} \Big| \sigma\big( t, \widetilde{X}_{N}(t) \big) - \sigma_N\big( t, \widetilde{X}_{N}(t) \big) \Big| \quad \xrightarrow[N \to \infty]{\Prob} \quad 0.
\end{equation}
To this end, we show $L_\theta$ convergence of the respective random variables to zero for appropriate values of $\theta \in (0,\infty)$. First, combining the condition~\hyperlink{ass:pL}{\normalfont (pL$_{r/2}^{\sigma}$)}, the Cauchy--Schwarz inequality, the triangle inequality, Proposition~\ref{prop:dist_X_tildeX}, the moment bound~\eqref{eq:appl_help2}, and Proposition \ref{prop:EsupTildeX} yields
\begin{equation*}
\begin{split}
&\bigg\Vert \sup_{t \in [0,T]} \left| \sigma\big( t, X(t) \big) - \sigma\big( t, \widetilde{X}_{N}(t) \big) \right| \bigg\Vert_{L_\theta} \\
&\le c \cdot \bigg\Vert \sup_{t \in [0,T]} \big| X(t) - \widetilde{X}_{N}(t) \big| \cdot \left( 1 + \big| X(t) \big|^{r/2} + \big| \widetilde{X}_{N}(t) \big|^{r/2} \right) \bigg\Vert_{L_\theta} \\
%&\le c \cdot \bigg( \E\bigg[ \Big\{ \sup_{t \in [0,T]} \big| X(t) - \widetilde{X}_{N}(t) \big|^\theta \Big\} \cdot \Big\{ \sup_{t \in [0,T]} \Big( 1 + \big| X(t) \big| + \big| \widetilde{X}_{N}(t) \big| \Big)^\theta \Big\} \bigg] \bigg)^{1/\theta} \\
&\le c \cdot \bigg\Vert \sup_{t \in [0,T]} \big| X(t) - \widetilde{X}_{N}(t) \big| \bigg\Vert_{L_{2\theta}} \cdot \bigg\Vert \sup_{t \in [0,T]} \left( 1 + \big| X(t) \big|^{r/2} + \big| \widetilde{X}_{N}(t) \big|^{r/2} \right) \bigg\Vert_{L_{2\theta}} \\
&\le c \cdot \bigg\Vert \sup_{t \in [0,T]} \big| X(t) - \widetilde{X}_{N}(t) \big| \bigg\Vert_{L_{2\theta}} \cdot \bigg( 1 + \bigg\Vert \sup_{t \in [0,T]} \big| X(t) \big|^{r/2} \bigg\Vert_{L_{2\theta}} + \bigg\Vert \sup_{t \in [0,T]} \big| \widetilde{X}_{N}(t) \big|^{r/2} \bigg\Vert_{L_{2\theta}} \bigg) \\
&\le c \cdot N^{-1/2}
\end{split}
\end{equation*}
for all $N \in \N$ where $\theta := \min\{ a, p/(2r+1) \}/3 \in [2/3,\infty)$. By letting $N$ tend to infinity, we eventually obtain \eqref{eq:ch3_inProb1}. Second, combining the growth condition~\hyperlink{ass:pG}{\normalfont (pG$_{(r+2)/2}^{\sigma}$)}, the triangle inequality, and Proposition~\ref{prop:EsupTildeX} yields
\begin{equation*}
\begin{split}
&\bigg\Vert \sup_{t \in [0,T]} \Big| \sigma\big( t, \widetilde{X}_{N}(t) \big) - \sigma_N\big( t, \widetilde{X}_{N}(t) \big) \Big| \bigg\Vert_{L_\theta} \\
&= (T/N)^{1/2} \cdot \left\Vert \sup_{t \in [0,T]} \left| \frac{\sigma\big( t, \widetilde{X}_{N}(t) \big) \cdot \big| \widetilde{X}_{N}(t) \big|^r}{1+(T/N)^{1/2} \cdot\big| \widetilde{X}_{N}(t) \big|^r} \right| \right\Vert_{L_\theta} \\
&\le (T/N)^{1/2} \cdot \bigg\Vert \sup_{t \in [0,T]} \big| \sigma\big( t, \widetilde{X}_{N}(t) \big) \big| \cdot \big| \widetilde{X}_{N}(t) \big|^{r} \bigg\Vert_{L_\theta} \\
&\le c \cdot N^{-1/2} \cdot \bigg\Vert \sup_{t \in [0,T]} \left( 1 + \big| \widetilde{X}_{N}(t) \big|^{(r+2)/2} \right) \cdot \big| \widetilde{X}_{N}(t) \big|^r \bigg\Vert_{L_\theta} \\
&\le c \cdot N^{-1/2} \cdot \bigg( \bigg\Vert \sup_{t \in [0,T]} \big| \widetilde{X}_{N}(t) \big|^{r} \bigg\Vert_{L_\theta} + \bigg\Vert \sup_{t \in [0,T]} \big| \widetilde{X}_{N}(t) \big|^{(3r+2)/2} \bigg\Vert_{L_\theta} \bigg) \\
%&\le 2 \cdot c \cdot N^{-1/2} \cdot \bigg( \E\bigg[ \sup_{t \in [0,T]} \Big( 1 + \big| \widetilde{X}_{N}(t) \big|^{(r+2) \cdot \theta} \Big) \bigg] + \E\bigg[ \sup_{t \in [0,T]} \Big( 1 + \big| \widetilde{X}_{N}(t) \big|^{r \cdot \theta} \Big) \bigg] \bigg)^{1/\theta} \\
%&\le c \cdot N^{-1/2} \cdot \bigg\{ \bigg( \E\bigg[ \sup_{t \in [0,T]} \Big( \big| \widetilde{X}_{N}(t) \big|^{\theta \cdot r} \Big) \bigg] \bigg)^{1/\theta} + \bigg( \E\bigg[ \sup_{t \in [0,T]} \Big( \big| \widetilde{X}_{N}(t) \big|^{\theta(3r+2)/2} \Big) \bigg] \bigg)^{1/\theta} \bigg\} \\
%&\le c \cdot N^{-1/2} \cdot \bigg( 1 + \bigg\Vert \sup_{t \in [0,T]} \big| \widetilde{X}_{N}(t) \big| \bigg\Vert_{L_{p-r}} \bigg)^{(3r+2)/2} \\
&\le c \cdot N^{-1/2}
\end{split}
\end{equation*}
for all $N \in \N$ where $\theta := 2 \cdot (p-r)/(3r+2) \in [2,\infty)$. By letting $N$ tend to infinity, we eventually obtain \eqref{eq:ch3_inProb2}.
From \eqref{eq:ch3_inProbability} we next conclude that
\begin{equation}
\label{eq:theorem_ad_convergence_in_Prob}
\bigg( \frac{T}{N} \cdot \sum_{\ell=0}^{N-1} \big\vert \sigma_N\big(t_\ell^{(N)}, \widetilde{X}_N(t_\ell^{(N)}) \big) \big\vert_{\infty,2}^2 \bigg)^{1/2} \quad \xrightarrow[N \to \infty]{\Prob} \quad \bigg( \int_0^T \big\vert \sigma\big( t, X(t) \big) \big\vert_{\infty,2}^2 \, \mathrm{d}t \bigg)^{1/2}
\end{equation}
and
\begin{equation}
\label{eq:theorem_eq_convergence_in_Prob}
\max_{\ell \in \{ 0, \ldots, N-1 \}} \big\vert \sigma_N\big(t_\ell^{(N)}, \widetilde{X}_N(t_\ell^{(N)}) \big) \big\vert_{\infty,2} \quad \xrightarrow[N \to \infty]{\Prob} \quad \sup_{t \in [0,T]} \big\vert \sigma\big( t, X(t) \big) \big\vert_{\infty,2}. 
\end{equation}
As a final step, we prove that
\begin{equation}
\label{eq:uniform_integrability_1}
\Bigg( \bigg( \frac{T}{N} \cdot \sum_{\ell=0}^{N-1} \big\vert \sigma_N\big(t_\ell^{(N)}, \widetilde{X}_N(t_\ell^{(N)}) \big) \big\vert_{\infty,2}^2 \bigg)^{q/2} \Bigg)_{N \in \N} \quad \text{is uniformly integrable}
\end{equation}
and
\begin{equation}
\label{eq:uniform_integrability_2}
\Big( \max_{\ell \in \{ 0, \ldots, N-1 \}} \big\vert \sigma_N\big(t_\ell^{(N)}, \widetilde{X}_N(t_\ell^{(N)}) \big) \big\vert_{\infty,2}^q \Big)_{N \in \N} \quad \text{is uniformly integrable}.
\end{equation}
To this end, observe that $1 \le q < 2(p-r)/(r+2)$, and hence the growth condition~\hyperlink{ass:pG}{\normalfont (pG$_{(r+2)/2}^{\sigma}$)} and Proposition~\ref{prop:EsupTildeX} give
\begin{equation*}
\begin{split}
&\sup_{N \in \N} \bigg\Vert \sup_{t \in [0,T]} \big\vert \sigma_N\big(t, \widetilde{X}_N(t) \big) \big\vert_{\infty,2} \bigg\Vert_{L_{2(p-r)/(r+2)}} \le c \cdot \bigg( 1 + \sup_{N \in \N} \bigg\Vert \sup_{t \in [0,T]} \big\vert \widetilde{X}_N(t) \big\vert_{\infty,2} \bigg\Vert_{L_{p-r}}^{(r+2)/2} \bigg) < \infty,
\end{split}
\end{equation*}
which immediately implies \eqref{eq:uniform_integrability_1} and \eqref{eq:uniform_integrability_2}.
Combining \eqref{eq:theorem_ad_convergence_in_Prob} and \eqref{eq:uniform_integrability_1} as well as \eqref{eq:theorem_eq_convergence_in_Prob} and \eqref{eq:uniform_integrability_2} then finally yields \eqref{eq:theorem_ad_convergence_in_Lq} and \eqref{eq:theorem_eq_convergence_in_Lq}. 

Applying Theorems~\ref{th:theorem_ad} and \ref{th:theorem_eq} finishes the proof of this corollary.
\end{proof}

We illustrate the results of Corollary~\ref{cor:tamedEuler} by a numerical experiment. 

\begin{Example}
Consider the introductory SDE \eqref{eq:SDE_HES} regarding the Heston--$3/2$--model with parameters $d = 1$, $m = 1$, $T = 1$, $\alpha = 5$, $\beta = 1$, $\gamma = 1$, and $\xi = 1$. This SDE thus reads as
\begin{equation}
\label{eq:SDE_HES_example}
\begin{split}
\text{d}X(t) &= 5 \cdot X(t) \cdot \big( 1 - |X(t)| \big) \, \mathrm{d}t + |X(t)|^{3/2} \, \mathrm{d}W(t), \quad t \in [0,1], \\
X(0) &= 1.
\end{split}
%\tag{HES}
%\tag{HES$_{T,\xi,\alpha,\beta,\gamma}$}
\end{equation}

We are interested in strongly asymptotically optimal approximations with respect to the error~$e_2$, i.e., we fix $q=2$.
It is easy to see that the SDE~\eqref{eq:SDE_HES_example} satisfies all the assumptions of Corollary~\ref{cor:tamedEuler}. More precisely, Assumption~\hyperlink{ass:I}{\normalfont (I$_{p}$)} is satisfied for all $p \in [0,\infty)$, Assumption~\hyperlink{ass:H}{\normalfont (H)} is satisfied, Assumption~\hyperlink{ass:K}{\normalfont (K$_{p}$)} is satisfied for all $p \in [2,11]$, Assumption~\hyperlink{ass:M}{\normalfont (M$_{a}$)} is satisfied for all $a \in [2,6]$, and Assumption~\hyperlink{ass:pL}{\normalfont (pL$_{r}^{\mu}$)} is satisfied for all $r \in [1,\infty)$. For the rest of this example, we fix $p=11$, $a=6$, and $r=1$. \linebreak Moreover, let the functions $(\mu_N)_{N \in \N}$ and $(\sigma_N)_{N \in \N}$ be given by~\eqref{eq:muN_sigmaN_tamedEuler}. As indicated in the beginning of this subsection, we hereinafter refer to the modified EM schemes as tamed Euler schemes.

In view of \eqref{eq:theorem_tamedSchemes_ad} and \eqref{eq:theorem_tamedSchemes_eq}, we aim at visualizing that, for large $N \in \N$, the approximation errors $e_2( \widehat{X}_{N,2}^{\mathrm{ad}} )$ and $e_2( \widehat{X}_{N}^{\mathrm{eq}} )$ of the adaptive and of the equidistant tamed Euler schemes are close to $C_2^{\mathrm{ad}} \cdot ( \log( c( \widehat{X}_{N,2}^{\mathrm{ad}} ) ) / c( \widehat{X}_{N,2}^{\mathrm{ad}} ) )^{1/2}$ and $C_2^{\mathrm{eq}} \cdot (\log( N ) / N)^{1/2}$, respectively.

We thereby encounter three different approximation issues, namely, the approximation of the asymptotic constants $C_2^{\mathrm{ad}}$ and $C_2^{\mathrm{eq}}$, of the errors $e_2( \widehat{X}_{N,2}^{\mathrm{ad}} )$ and $e_2( \widehat{X}_{N}^{\mathrm{eq}} )$, and of the average number of evaluations $c( \widehat{X}_{N,2}^{\mathrm{ad}} )$.

Regarding the first approximation issue, we do not know numerically suitable closed-form expressions of the constants $C_2^{\mathrm{ad}}$ and $C_2^{\mathrm{eq}}$, nor of the solution, for the particular SDE \eqref{eq:SDE_HES_example}. Therefore, we estimate these constants via Monte Carlo simulations in which we approximate the solution by an equidistant tamed Euler scheme with a sufficiently large number of discretization points. More precisely, we estimate $C_2^{\mathrm{ad}}$ and $C_2^{\mathrm{eq}}$ by
\begin{equation*}
\widehat{C}_{2,M,N}^{\mathrm{ad}} := 2^{-1/2} \cdot \frac{1}{M} \cdot \sum_{m=1}^M \bigg( \frac{1}{N} \cdot \sum_{\ell=0}^{N-1} \big| \widehat{X}_{N,m}^{\mathrm{eq}}(t_\ell^{(N)}) \big|^3 \bigg)^{1/2}
\end{equation*}
and
\begin{equation*}
\widehat{C}_{2,M,N}^{\mathrm{eq}} := 2^{-1/2} \cdot \bigg( \frac{1}{M} \cdot \sum_{m=1}^M \max_{\ell \in \{ 0, \ldots, N \}} \big| \widehat{X}_{N,m}^{\mathrm{eq}}(t_\ell^{(N)}) \big|^3 \bigg)^{1/2},
\end{equation*}
respectively, where $M, N \in \N$ and where the random vectors
\begin{equation*}
\begin{split}
&\big( \widehat{X}_{N,m}^{\mathrm{eq}}(t_0^{(N)}), \ldots, \widehat{X}_{N,m}^{\mathrm{eq}}(t_N^{(N)}) \big), \quad m \in \{ 1, \ldots, M \},
\end{split}
\end{equation*}
are independent copies of $( \widehat{X}_{N}^{\mathrm{eq}}(t_0^{(N)}), \ldots, \widehat{X}_{N}^{\mathrm{eq}}(t_N^{(N)}) )$. Observe that for $C_2^{\mathrm{ad}}$, we approximate the integral occurring in its definition by left Riemann sums. Proposition \ref{prop:dist_X_tildeX} in Appendix \ref{sec:Appendix} implies that $\widehat{C}_{2,M,N}^{\mathrm{ad}}$ and $\widehat{C}_{2,M,N}^{\mathrm{eq}}$ tend to $C_2^{\mathrm{ad}}$ and $C_2^{\mathrm{eq}}$, respectively, as $M$ and $N$ tend to infinity. Figure \ref{figure:MC_constants} depicts simulations of $\widehat{C}_{2,M,2^{27}}^{\mathrm{ad}}$ and $\widehat{C}_{2,M,2^{27}}^{\mathrm{eq}}$ in dependence of $M$ along with their corresponding $95 \%$ CLT-based confidence intervals. Furthermore, we utilize the specific approximations $C_2^{\mathrm{ad}} \approx 0.7080$ and $C_2^{\mathrm{eq}} \approx 1.7749$ obtained from realizations of $\widehat{C}_{2,10^4,2^{27}}^{\mathrm{ad}}$ and of $\widehat{C}_{2,10^4,2^{27}}^{\mathrm{eq}}$, respectively, for the black lines featured in Figure \ref{figure:MC_errors}.

\begin{figure}[H]
\begin{center}
\input{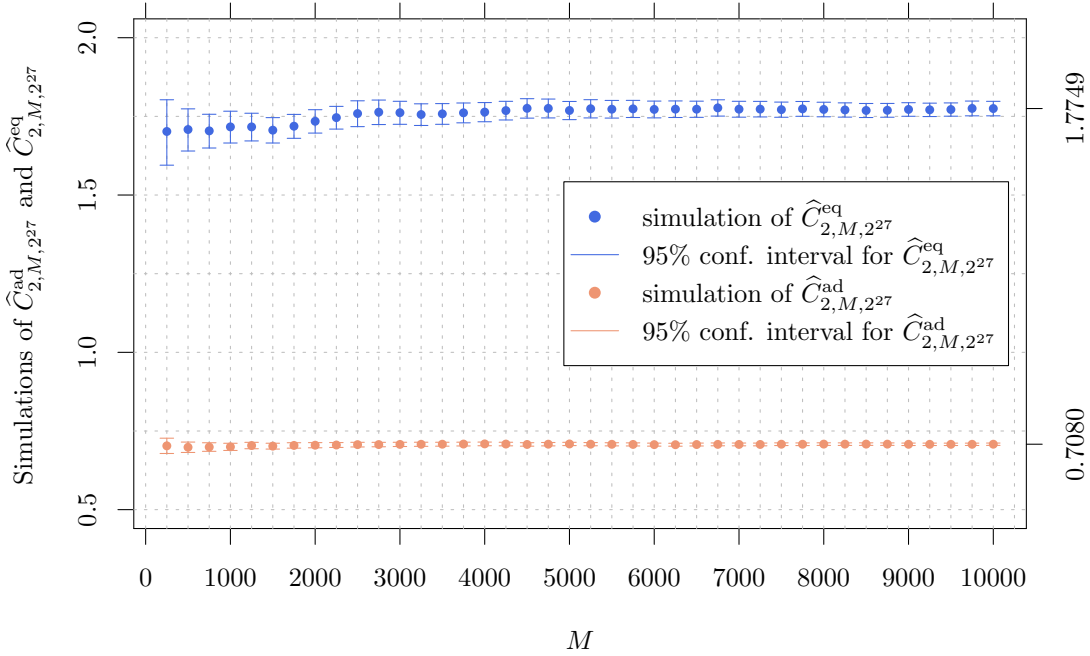}
\end{center}
\caption{Monte Carlo approximations of the asymptotic constants $C_2^{\mathrm{ad}}$ and $C_2^{\mathrm{eq}}$ for the SDE \eqref{eq:SDE_HES_example}.}
\label{figure:MC_constants}
\end{figure}

The remaining two approximation issues are addressed simultaneously. Similarly to the approximation of the asymptotic constants, we again estimate the solution by a sufficiently accurate equidistant tamed Euler scheme, and we approximate the errors of the equidistant tamed Euler schemes as well as the errors and the average numbers of evaluations of the adaptive tamed Euler schemes via Monte Carlo simulations. More precisely, for each $N \in \N$ we estimate $e_2(\widehat{X}_{N}^{\mathrm{eq}})$, $e_2(\widehat{X}_{N,2}^{\mathrm{ad}})$, and $c( \widehat{X}_{N,2}^{\mathrm{ad}} )$ by
\begin{equation*}
\widehat{e}_{2,M,N^*,N}^{\mathrm{eq}} := \bigg( \frac{1}{M} \cdot \sum_{m=1}^M \max_{\ell \in \{ 0, \ldots, N^* \}} \big| \widehat{X}_{N^*,m}^{\mathrm{eq}}(t_\ell^{(N^*)}) - \widehat{X}_{N,m}^{\mathrm{eq}}(t_\ell^{(N^*)}) \big|^2 \bigg)^{1/2},
\end{equation*}
\begin{equation*}
\widehat{e}_{2,M,N^*,N}^{\mathrm{ad}} := \bigg( \frac{1}{M} \cdot \sum_{m=1}^M \max_{\ell \in \{ 0, \ldots, N^* \}} \big| \widehat{X}_{N^*,m}^{\mathrm{eq}}(t_\ell^{(N^*)}) - \widehat{X}_{N,2,m}^{\mathrm{ad}}(t_\ell^{(N^*)}) \big|^2 \bigg)^{1/2},
\end{equation*}
and
\[ \widehat{c}_{M,N} := \frac{1}{M} \cdot \sum_{m=1}^M \nu_{N,2,m}^{\mathrm{ad}}, \]
respectively, where $M,N^* \in \N$ and where the random vectors
\begin{equation*}
\begin{split}
&\big( \widehat{X}_{N^*,m}^{\mathrm{eq}}(t_0^{(N^*)}), \ldots, \widehat{X}_{N^*,m}^{\mathrm{eq}}(t_{N^*}^{(N^*)}) \big), \quad m \in \{ 1, \ldots, M \},
\end{split}
\end{equation*}
are independent copies of $( \widehat{X}_{N^*}^{\mathrm{eq}}(t_0^{(N^*)}), \ldots, \widehat{X}_{N^*}^{\mathrm{eq}}(t_{N^*}^{(N^*)}) )$, the random vectors
\begin{equation*}
\begin{split}
&\big( \widehat{X}_{N,m}^{\mathrm{eq}}(t_0^{(N^*)}), \ldots, \widehat{X}_{N,m}^{\mathrm{eq}}(t_{N^*}^{(N^*)}) \big), \quad m \in \{ 1, \ldots, M \},
\end{split}
\end{equation*}
are independent copies of $( \widehat{X}_{N}^{\mathrm{eq}}(t_0^{(N^*)}), \ldots, \widehat{X}_{N}^{\mathrm{eq}}(t_{N^*}^{(N^*)}) )$, the random vectors
\begin{equation*}
\begin{split}
&\big( \widehat{X}_{N,2,m}^{\mathrm{ad}}(t_0^{(N^*)}), \ldots, \widehat{X}_{N,2,m}^{\mathrm{ad}}(t_{N^*}^{(N^*)}) \big), \quad m \in \{ 1, \ldots, M \},
\end{split}
\end{equation*}
are independent copies of $( \widehat{X}_{N,2}^{\mathrm{ad}}(t_0^{(N^*)}), \ldots, \widehat{X}_{N,2}^{\mathrm{ad}}(t_{N^*}^{(N^*)}) )$, and the random variables
\[ \nu_{N,2,m}^{\mathrm{ad}}, \quad m \in \{ 1, \ldots, M \}, \]
are independent copies of $\nu_{N,2}^{\mathrm{ad}}$.
For the adaptive tamed Euler schemes, we used $k_N := \linebreak \lceil N \cdot (\log(N+1))^{-1/2} \rceil$ for all $N \in \N$ on every computation.
Numerical estimates $(N,\widehat{e}_{2,10^4,2^{27},N}^{\mathrm{eq}})$, $N \in \{ 2^6, 2^8, \ldots, 2^{20} \}$, and $(\widehat{c}_{10^3,N},\widehat{e}_{2,10^3,2^{27},N}^{\mathrm{ad}})$, $N \in \{ 2^7, 2^9, \ldots, 2^{21} \}$, are visualized in Figure \ref{figure:MC_errors}. \hfill $\Diamond$

\begin{figure}[H]
\begin{center}
\input{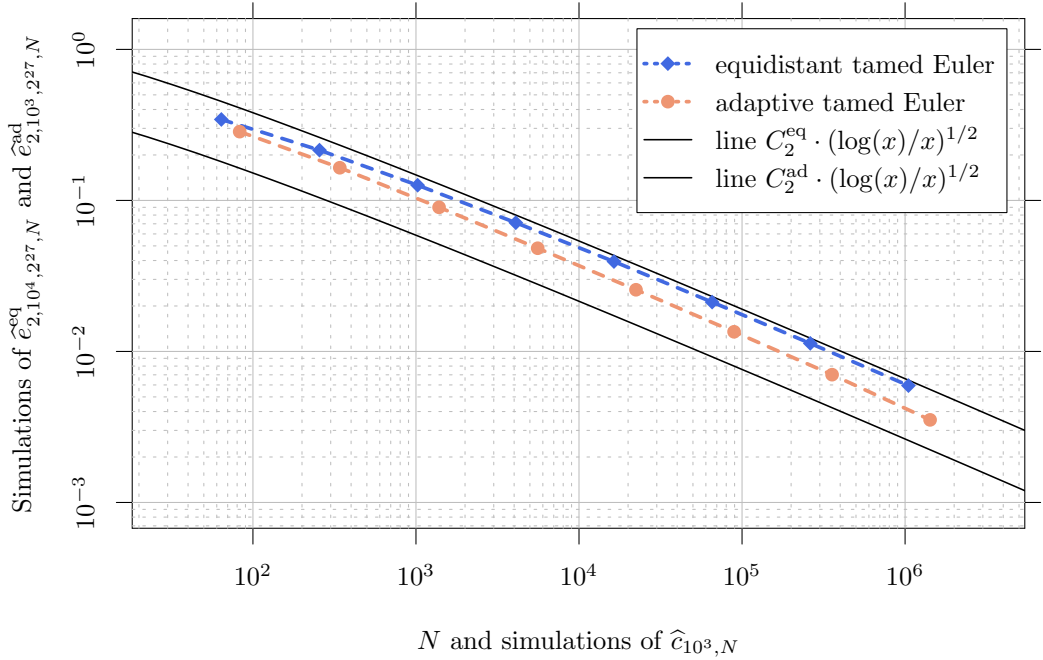}
\end{center}
\caption{Monte Carlo approximations of the errors $e_2(\widehat{X}_{N}^{\mathrm{eq}})$ and $e_2(\widehat{X}_{N,2}^{\mathrm{ad}})$ versus $N$ and Monte Carlo approximations of the average number of evaluations $c(\widehat{X}_{N,2}^{\mathrm{ad}})$ for the SDE \eqref{eq:SDE_HES_example}.}
\label{figure:MC_errors}
\end{figure}

\end{Example}

%\newpage
%%%%%%%%%%%%%%%%%%%%%%%%%%%%%%%%%%%%%%%%%%%%%%%%%%%%%%
\section{Proofs}
\label{sec:Proofs}
%%%%%%%%%%%%%%%%%%%%%%%%%%%%%%%%%%%%%%%%%%%%%%%%%%%%%%

In the following, we prove the theorems presented in Section~\ref{sec:MainTheorems} by showing asymptotic lower bounds relating to \eqref{eq:theorem_minErrors_ad} and \eqref{eq:theorem_minErrors_eq} as well as asymptotic upper bounds relating to \eqref{eq:theorem_tamedSchemes_ad} and \eqref{eq:theorem_tamedSchemes_eq}. The structure of the corresponding proofs is to a large extent based on techniques developed in Müller-Gronbach~\cite{tmg2002}.
 
Throughout this section, let $(\mu_N)_{N \in \N}$ and $(\sigma_N)_{N \in \N}$ be measurable functions as per Section~\ref{sec:Schemes}. In addition, let $(k_N)_{N \in \N}$ be a sequence of natural numbers satisfying the limits~\eqref{eq:limits_kN}, and let $c$ denote unspecified positive constants that may vary at every occurrence and that may only depend on $T$, $d$, $m$, and the parameters and constants from assumptions used in the respective lemmas.

%\subsection{Preliminaries}
%
%As a first step, we prove that the average numbers of evaluations of the driving Brownian motion of the adaptive tamed Euler schemes are positive, finite, and tend to infinity as the number of discretization sites tends to infinity. We obtain, in particular, that each such approximation does indeed lie in one of the classes of adaptive schemes.
%\begin{Lemma}
%\label{lemma:adaptiveTE_in_adaptiveClass}
%For all $q \in [1,\min\{ a, p/(2r+1) \})$ and for all $N \in \N$ it holds that
%\begin{equation}
%\label{eq:nichtderredewert}
%k_N \le c\big( \widehat{X}_{N,q}^{\mathrm{ad}} \big) < \infty.
%\end{equation}
%\end{Lemma}
%
%\begin{proof}
%It is easy to see that the inequalities \eqref{eq:nichtderredewert} immediately follow from the estimates \eqref{eq:inequality_totalNumber_nachOben} and \eqref{eq:inequality_totalNumber_nachUnten} of the total number of evaluations of $W$ of the respective adaptive tamed Euler scheme, provided that for the proof of the upper bound one also employs Assumption \hyperlink{ass:pG}{\normalfont (pG$_{(r+2)/2}^\sigma$)} along with Proposition \ref{prop:EsupTildeX} from Appendix \ref{sec:Appendix}.
%\end{proof}
%

For the convenience of the reader, we also provide a lemma containing a simple subsequence argument that will be employed in the proofs of the Lemmas~\ref{lemma:liminf_ad} and \ref{lemma:liminf_eq}.

\begin{Lemma}
\label{lemma:appendixB_liminf_subsequence_argument}
Let $(a_N)_{N \in \N}$ be a sequence of real numbers that is bounded from below and let $C \in \R$. Then the following are equivalent:
\begin{itemize}
\item[\textit{(i)}]  It holds that $\liminf_{N \to \infty} a_N \ge C$.
\item[\textit{(ii)}] For every subsequence $(a_{N_\kappa})_{\kappa \in \N}$ of $(a_N)_{N \in \N}$ there exists a subsequence $(a_{N_{\kappa_n}})_{n \in \N}$ of \linebreak $(a_{N_\kappa})_{\kappa \in \N}$ such that $\liminf_{n \to \infty} a_{N_{\kappa_n}} \ge C$.
\end{itemize}
\end{Lemma}

%\newpage

\subsection{Asymptotic lower bounds}

We start by introducing some notation that will be used in this subsection. For all $q \in [1,\infty)$, for all $N \in \N$, for all $\alpha_1, \ldots, \alpha_N \in [0,\infty)$, and for all independent real-valued Brownian bridges $B_1, \ldots, B_N$ on $[0,1]$ from $0$ to $0$ we put
\[ \mathcal{M}_q(\alpha_1, \ldots, \alpha_N) := \E\bigg[ \max_{\ell \in \{ 1, \ldots, N \}} \Big( \alpha_\ell \cdot \sup_{t \in [0,1]} |B_\ell(t)| \Big)^q \bigg] \in [0,\infty) \]
and
\[ \mathcal{M}_q(N) := \mathcal{M}_q(\underbrace{1, \ldots, 1}_{N \text{ times}}). \]

First, we prove an asymptotic lower bound for the $N$th minimal errors in the classes of adaptive approximations.

\begin{Lemma}
\label{lemma:liminf_ad}
Assume the setting of Theorem~\ref{th:theorem_ad}. Then it holds that
\begin{equation}
\label{eq:ear}
\liminf_{N \to \infty} \big( N / \log(N) \big)^{1/2} \cdot \inf\Big\{ e_q\big( \widehat{X} \big) \; \Big| \; \widehat{X} \in \mathbb{X}_N^{\mathrm{ad}} \Big\} \ge C_q^{\mathrm{ad}}.
\end{equation}
\end{Lemma}

\begin{proof}
Fix $N \in \N$ with $N > \exp(2)$ and $\widehat{X}_N \in \mathbb{X}_N^{\mathrm{ad}}$ for the moment. Due to the inverse triangle inequality and assumption~\eqref{eq:theorem_ad_strong_convergence}, it holds that
\begin{equation}
\label{eq:hammer}
\begin{split}
e_q\big( \widehat{X}_N \big) &= \Big\Vert \big\Vert X - \widehat{X}_N \big\Vert_\infty \Big\Vert_{L_q} \ge \Big\Vert \big\Vert \widetilde{X}_{k_N} - \widehat{X}_N \big\Vert_\infty \Big\Vert_{L_q} - c \cdot k_N^{-1/2}.
\end{split}
\end{equation}
Let $D_N$ denote the entire data used by $\widehat{X}_N$ in the sense of Section~\ref{sec:Classes}, define $\Psi_N$ to be the set of observation sites of the driving Brownian motion employed in $\widehat{X}_N$, and put $\nu_N := \# \Psi_N$. As a first step, we show that the distance between $\widetilde{X}_{k_N}$ and $\widehat{X}_N$ as above is greater or equal than the respective distance between $\widetilde{X}_{k_N}$ and $\E[ \widehat{X}_{k_N} \, | \, D_N ]$. Because of the first limit in \eqref{eq:limits_kN}, we may actually assume that $\{ t_1^{(k_N)}, \ldots, t_{k_N}^{(k_N)} \} \subseteq \Psi_N$. Hence, each $\widetilde{X}_{k_N}(t_\ell^{(k_N)})$, $\ell \in \{0, \ldots, k_N \}$, is measurable with respect to the $\sigma$-algebra generated by $D_N$ and we thereby obtain
\begin{equation}
\label{eq:newark}
\begin{split}
\widetilde{X}_{k_N}(t) - \E\big[ \widetilde{X}_{k_N}(t) \, \big| \, D_N \big]
%&= \frac{\sigma\big( t_\ell^{(N)}, \widetilde{X}_N(t_\ell^{(N)}) \big)}{1+(T/N)^{1/2} \cdot \big| \widetilde{X}_N(t_\ell^{(N)}) \big|^r} \cdot \left( W(t) - \left[ \frac{t_{\ell+1}^{(N)} - t}{t_{\ell+1}^{(N)} - t_\ell^{(N)}} \cdot W(t_\ell^{(N)}) + \frac{t - t_\ell^{(N)}}{t_{\ell+1}^{(N)} - t_\ell^{(N)}} \cdot W(t_{\ell+1}^{(N)}) \right] \right) \\
&= \sigma_{k_N}\big( t_\ell^{(k_N)}, \widetilde{X}_{k_N}(t_\ell^{(k_N)}) \big) \cdot \Big( W(t) - \E\big[ W(t) \, \big| \, D_N \big] \Big)
\end{split}
\end{equation}
for all $\ell \in \{ 0, \ldots, k_N-1 \}$ and for all $t \in (t_\ell^{(k_N)},t_{\ell+1}^{(k_N)}]$. Similarly to the derivations of the Lemmas~1 and 2 in Yaroslavtseva~\cite{yaroslavtseva2017}, one shows that for $\Prob^{D_N}$-almost all $(x,y) \in \R^d \times \bigcup_{n \in \N} \R^m$ it holds that
\[ \Prob^{W | D_N = (x,y)} = \Prob^{-W | D_N = (x,y)}, \]
which along with \eqref{eq:newark} yields
\[ \Prob^{\widetilde{X}_{k_N} - \E[ \widetilde{X}_{k_N} | D_N ] \, | \, D_N = (x,y)} = \Prob^{-\widetilde{X}_{k_N} + \E[ \widetilde{X}_{k_N} | D_N ] \, | \, D_N = (x,y)}. \]
We thus conclude that $(\widetilde{X}_{k_N} - \E[ \widetilde{X}_{k_N} \, | \, D_N ], \, D_N)$ and $(-\widetilde{X}_{k_N} + \E[ \widetilde{X}_{k_N} \, | \, D_N ], \, D_N)$ are identically distributed. Since, additionally, both $\widehat{X}_N$ and $\E[ \widetilde{X}_{k_N} \, | \, D_N ]$ are measurable functions of $D_N$, we consequently find that $\widetilde{X}_{k_N} - \widehat{X}_N$ and $2 \E[ \widetilde{X}_{k_N} \, | \, D_N ] - \widetilde{X}_{k_N} - \widehat{X}_N$ are also identically distributed. Therefore, we obtain
%\begin{equation*}
%\Big\Vert \big\Vert \widetilde{X}_{k_N} - \widehat{X}_N \big\Vert_{L_\infty([0,T])} \Big\Vert_{L_q(\Omega)} = \Big\Vert \big\Vert 2 \cdot \E\big[ \widetilde{X}_{k_N} \, \big| \, D_N \big] - \widetilde{X}_{k_N} - \widehat{X}_N \big\Vert_{L_\infty([0,T])} \Big\Vert_{L_q(\Omega)} \\
%\end{equation*}
%and hence
\begin{equation}
\label{eq:amboss}
\begin{split}
&\Big\Vert \big\Vert \widetilde{X}_{k_N} - \widehat{X}_N \big\Vert_\infty \Big\Vert_{L_q} \\
&= 1/2 \cdot \bigg( \Big\Vert \big\Vert \widetilde{X}_{k_N} - \widehat{X}_N \big\Vert_\infty \Big\Vert_{L_q} + \Big\Vert \big\Vert -2 \E[ \widetilde{X}_{k_N} \, | \, D_N ] + \widetilde{X}_{k_N} + \widehat{X}_N \big\Vert_\infty \Big\Vert_{L_q} \bigg) \\
&\ge 1/2 \cdot \Big\Vert \big\Vert \widetilde{X}_{k_N} - \widehat{X}_N -2 \E[ \widetilde{X}_{k_N} \, | \, D_N ] + \widetilde{X}_{k_N} + \widehat{X}_N \big\Vert_\infty \Big\Vert_{L_q} \\
&= \Big\Vert \big\Vert \widetilde{X}_{k_N} - \E\big[ \widetilde{X}_{k_N} \, \big| \, D_N \big] \big\Vert_\infty \Big\Vert_{L_q}.
\end{split}
\end{equation}
%\begin{equation*}
%\begin{split}
%&\Big\Vert \big\Vert \widetilde{X}_{k_N} - \E\big[ \widetilde{X}_{k_N} \, \big| \, D_N \big] \big\Vert_\infty \Big\Vert_{L_q}
%%&= \frac{1}{2} \cdot \Big\Vert \big\Vert 2 \cdot \Big( \widetilde{X}_{k_N} - \E\big[ \widetilde{X}_{k_N} \, \big| \, D_N \big] \Big) \big\Vert_\infty \Big\Vert_{L_q(\Prob,| \cdot |)} \\
%%&= \frac{1}{2} \cdot \Big\Vert \big\Vert \widetilde{X}_{k_N} - \widehat{X}_N_N - 2 \cdot \E\big[ \widetilde{X}_{k_N} \, \big| \, D_N \big] + \widetilde{X}_{k_N} + \widehat{X}_N \big\Vert_\infty \Big\Vert_{L_q(\Prob,| \cdot |)} \\
%%&\le \frac{1}{2} \cdot \left( \Big\Vert \big\Vert \widetilde{X}_{k_N} - \widehat{X}_N \big\Vert_\infty \Big\Vert_{L_q(\Prob,| \cdot |)} + \Big\Vert \big\Vert - 2 \cdot \E\big[ \widetilde{X}_{k_N} \, \big| \, D_N \big] + \widetilde{X}_{k_N} + \widehat{X}_N \big\Vert_\infty \Big\Vert_{L_q(\Prob,| \cdot |)} \right) \\
%%&= \frac{1}{2} \cdot \left( 2 \cdot \Big\Vert \big\Vert \widetilde{X}_{k_N} - \widehat{X}_N \big\Vert_\infty \Big\Vert_{L_q(\Prob,| \cdot |)} \right) \\
%\le \Big\Vert \big\Vert \widetilde{X}_{k_N} - \widehat{X} \big\Vert_\infty \Big\Vert_{L_q}.
%\end{split}
%\end{equation*}
Almost identically to the proof of inequality (12) in Müller-Gronbach~\cite{tmg2002} and the ensuing inequality therein, one subsequently shows that
\begin{equation}
\label{eq:slice}
\E\Big[ \big\Vert \widetilde{X}_{k_N} - \E\big[ \widetilde{X}_{k_N} \, \big| \, D_N \big] \big\Vert_\infty^q \; \Big| \; D_N \Big] \ge \A_{k_N}^q \cdot \delta_N^{-q/2} \cdot \mathcal{M}_q(\delta_N)
\end{equation}
holds almost surely where $\A_{k_N}$ is defined as in \eqref{eq:tildeA_kN} and where
\[ \delta_N := \max\bigg\{ 1, \sum_{\ell \in L_N} \Big( \#\big( \Psi_N \cap (t_\ell^{(k_N)}, t_{\ell+1}^{(k_N)}) \big) +1 \Big) \bigg\} \]
with
\[ L_N := \left\{ \ell \in \{ 0, \ldots, k_N-1 \} \; \middle| \; \big\vert \sigma_{k_N}\big( t_\ell^{(k_N)}, \widetilde{X}_{k_N}(t_\ell^{(k_N)}) \big) \big\vert_{\infty,2} > 0 \right\}. \]
By using arguments in a similar way to the ones in the proof of the last inequality on page 681 in Müller-Gronbach~\cite{tmg2002}, we arrive at
\begin{equation}
\label{eq:steigbuegel}
\begin{split}
&\big( N/ \log(N) \big)^{1/2} \cdot \left\Vert \A_{k_N} \cdot \delta_N^{-1/2} \cdot \big( \mathcal{M}_q(\delta_N) \big)^{1/q} \right\Vert_{L_q} \\
&\ge \left\Vert \mathcal{A}_{k_N} \cdot \big( \log(\delta_N) \big)^{-1/2} \cdot \big( \mathcal{M}_q(\delta_N) \big)^{1/q} \cdot \mathds{1}_{\{ \delta_N > \exp(2) \} \cap \{ \int_0^T \vert \sigma(t, X(t) ) \vert_{\infty,2}^2 \, \mathrm{d}t > 0 \}} \right\Vert_{L_{2q/(q+2)}}.
\end{split}
\end{equation}

Combining \eqref{eq:hammer}, \eqref{eq:amboss}, \eqref{eq:slice}, \eqref{eq:steigbuegel}, and the second limit in~\eqref{eq:limits_kN} yields
\begin{equation}
\label{eq:almost_ear}
\begin{split}
&\liminf_{N \to \infty} \big( N/ \log(N) \big)^{1/2} \cdot \inf\Big\{ e_q\big( \widehat{X} \big) \; \Big| \; \widehat{X} \in \mathbb{X}_N^{\mathrm{ad}} \Big\} \ge \liminf_{N \to \infty} \big\Vert \alpha^{(N)} \big\Vert_{L_{2q/(q+2)}} \\
\end{split}
\end{equation}
where
\[ \alpha^{(N)} := \mathcal{A}_{k_N} \cdot \big( \log(\delta_N) \big)^{-1/2} \cdot \big( \mathcal{M}_q(\delta_N) \big)^{1/q} \cdot \mathds{1}_{\{ \delta_N > \exp(2) \} \cap \{ \int_0^T \vert \sigma(t, X(t) ) \vert_{\infty,2}^2 \, \mathrm{d}t > 0 \}}. \]

Next, we use the subsequence argument that is provided by Lemma \ref{lemma:appendixB_liminf_subsequence_argument} to conclude \eqref{eq:ear} from~\eqref{eq:almost_ear}. First of all, assumption~\eqref{eq:theorem_ad_convergence_in_Lq} implies
\begin{equation}
\label{eq:laeppchen}
\mathcal{A}_{k_N} \quad \xrightarrow[N \to \infty]{\Prob} \quad \bigg( \int_0^T \big\vert \sigma\big( t, X(t) \big) \big\vert_{\infty,2}^2 \, \mathrm{d}t \bigg)^{1/2}.
\end{equation}
Now let $(\alpha^{(N_\kappa)})_{\kappa \in \N}$ be a subsequence of $(\alpha^{(N)})_{N \in \N}$. In view of \eqref{eq:laeppchen}, there exists a subsequence $(\mathcal{A}_{k_{N_{\kappa_n}}})_{n \in \N}$ of $(\mathcal{A}_{k_{N_\kappa}})_{\kappa \in \N}$ such that
\begin{equation}
\label{eq:laeppchen_almostsurely}
\mathcal{A}_{k_{N_{\kappa_n}}} \quad \xrightarrow[n \to \infty]{\text{a.s.}} \quad \bigg( \int_0^T \big\vert \sigma\big( t, X(t) \big) \big\vert_{\infty,2}^2 \, \mathrm{d}t \bigg)^{1/2}.
\end{equation}
Some tedious calculations using \eqref{eq:laeppchen_almostsurely} show
\begin{equation}
\label{eq:ohrmuschel}
\begin{split}
&\Prob\bigg( \Big\{ \lim_{n \to \infty} \delta_{N_{\kappa_n}} = \infty \Big\} \cap \bigg\{ \int_0^T \big\vert \sigma\big( t, X(t) \big) \big\vert_{\infty,2}^2 \, \mathrm{d}t > 0 \bigg\} \bigg) \\
&= \Prob\bigg( \bigg\{ \int_0^T \big\vert \sigma\big( t, X(t) \big) \big\vert_{\infty,2}^2 \, \mathrm{d}t > 0 \bigg\} \bigg).
\end{split}
\end{equation}
On the one hand, we clearly have
\begin{equation}
\label{eq:keineAhnung1}
\begin{split}
&\Prob\bigg( \bigg\{ \liminf_{n \to \infty} \alpha^{(N_{\kappa_n})} \ge 2^{-1/2} \cdot \bigg( \int_0^T \big\vert \sigma\big( t, X(t) \big) \big\vert_{\infty,2}^2 \, \mathrm{d}t \bigg)^{1/2} \bigg\} \\
&\hspace{0.55cm}\cap \bigg\{ \int_0^T \big\vert \sigma\big( t, X(t) \big) \big\vert_{\infty,2}^2 \, \mathrm{d}t = 0 \bigg\} \bigg) \\
&= \Prob\bigg( \bigg\{ \int_0^T \big\vert \sigma\big( t, X(t) \big) \big\vert_{\infty,2}^2 \, \mathrm{d}t = 0 \bigg\} \bigg). \\
\end{split}
\end{equation}
On the other hand, the limit \eqref{eq:laeppchen_almostsurely}, an easy generalization of Corollary~2 in Müller-Gronbach~\cite{tmg2002} regarding non-negative instead of strictly positive scalars, and \eqref{eq:ohrmuschel} yield
\begin{equation}
\label{eq:keineAhnung2}
\begin{split}
&\Prob\bigg( \bigg\{ \liminf_{n \to \infty} \alpha^{(N_{\kappa_n})} \ge 2^{-1/2} \cdot \bigg( \int_0^T \big\vert \sigma\big( t, X(t) \big) \big\vert_{\infty,2}^2 \, \mathrm{d}t \bigg)^{1/2} \bigg\} \\
&\hspace{0.55cm}\cap \bigg\{ \int_0^T \big\vert \sigma\big( t, X(t) \big) \big\vert_{\infty,2}^2 \, \mathrm{d}t > 0 \bigg\} \bigg) \\
&= \Prob\bigg( \bigg\{ \liminf_{n \to \infty} \alpha^{(N_{\kappa_n})} \ge 2^{-1/2} \cdot \bigg( \int_0^T \big\vert \sigma\big( t, X(t) \big) \big\vert_{\infty,2}^2 \, \mathrm{d}t \bigg)^{1/2} \bigg\} \\
&\hspace{0.95cm}\cap \bigg\{ \int_0^T \big\vert \sigma\big( t, X(t) \big) \big\vert_{\infty,2}^2 \, \mathrm{d}t > 0 \bigg\} \cap \bigg\{ \lim_{n \to \infty} \mathcal{A}_{k_{N_{\kappa_n}}} = \bigg( \int_0^T \big\vert \sigma\big( t, X(t) \big) \big\vert_{\infty,2}^2 \, \mathrm{d}t \bigg)^{1/2} \bigg\} \bigg) \\
&\ge \Prob\bigg( \Big\{ \lim_{n \to \infty} \delta_{N_{\kappa_n}} = \infty \Big\} \\
&\hspace{0.95cm}\cap \bigg\{ \int_0^T \big\vert \sigma\big( t, X(t) \big) \big\vert_{\infty,2}^2 \, \mathrm{d}t > 0 \bigg\} \cap \bigg\{ \lim_{n \to \infty} \mathcal{A}_{k_{N_{\kappa_n}}} = \bigg( \int_0^T \big\vert \sigma\big( t, X(t) \big) \big\vert_{\infty,2}^2 \, \mathrm{d}t \bigg)^{1/2} \bigg\} \bigg) \\
&= \Prob\bigg( \bigg\{ \int_0^T \big\vert \sigma\big( t, X(t) \big) \big\vert_{\infty,2}^2 \, \mathrm{d}t > 0 \bigg\} \bigg). \\
\end{split}
\end{equation}
Combining \eqref{eq:keineAhnung1} and \eqref{eq:keineAhnung2}, we conclude that
\[ \liminf_{n \to \infty} \alpha^{(N_{\kappa_n})} \ge 2^{-1/2} \cdot \bigg( \int_0^T \big\vert \sigma\big( t, X(t) \big) \big\vert_{\infty,2}^2 \, \mathrm{d}t \bigg)^{1/2} \]
holds almost surely. Consequently, Fatou's lemma gives
\begin{equation*}
\liminf_{n \to \infty} \big\Vert \alpha^{(N_{\kappa_n})} \big\Vert_{L_{2q/(q+2)}} \ge C_q^{\mathrm{ad}}.
\end{equation*}

Finally, employing Lemma \ref{lemma:appendixB_liminf_subsequence_argument} finishes the proof of this lemma.
\end{proof}

%\newpage

Next, we prove an asymptotic lower bound for the $N$th minimal errors in the classes of equidistant approximations.

\begin{Lemma}
\label{lemma:liminf_eq}
Assume the setting of Theorem~\ref{th:theorem_eq}. Then it holds that
\begin{equation}
\label{eq:liminf_eq}
\liminf_{N \to \infty} \big( N/ \log(N) \big)^{1/2} \cdot \inf\Big\{ e_q\big( \widehat{X} \big) \; \Big| \; \widehat{X} \in \mathbb{X}_N^{\mathrm{eq}} \Big\} \ge C_q^{\mathrm{eq}}.
\end{equation}
\end{Lemma}

\begin{proof}
Fix $N \in \N$ and $\widehat{X}_N \in \mathbb{X}_N^{\mathrm{eq}}$ for the moment, and let $D_N := \big( \xi, W(t_1^{(N)}), \ldots, W(t_N^{(N)}) \big)$ denote the data used by $\widehat{X}_N$.
Similarly to the estimates \eqref{eq:hammer}, \eqref{eq:amboss}, and \eqref{eq:slice} in the proof of Lemma~\ref{lemma:liminf_ad}, one successively shows that\begin{equation}
\label{eq:serve}
\begin{split}
e_q\big( \widehat{X}_N \big) &= \Big\Vert \big\Vert X - \widehat{X}_N \big\Vert_\infty \Big\Vert_{L_q} 
%\le \Big\Vert \big\Vert \widetilde{X}_N - \widehat{X}_N^{eq} \big\Vert_\infty \Big\Vert_{L_q} + \Big\Vert \big\Vert X - \widetilde{X}_N \big\Vert_\infty \Big\Vert_{L_q} \\
\ge \Big\Vert \big\Vert \widetilde{X}_N - \widehat{X}_N \big\Vert_\infty \Big\Vert_{L_q} - c \cdot N^{-1/2}, \\
\end{split}
\end{equation}
\begin{equation}
\label{eq:topspin}
\Big\Vert \big\Vert \widetilde{X}_N - \widehat{X}_N \big\Vert_\infty \Big\Vert_{L_q} \ge \Big\Vert \big\Vert \widetilde{X}_N - \E\big[ \widetilde{X}_N \, \big| \, D_N \big] \big\Vert_\infty \Big\Vert_{L_q},
\end{equation}
and that
\begin{equation}
\label{eq:slice2}
\begin{split}
&\E\Big[ \big\Vert \widetilde{X}_N - \E\big[ \widetilde{X}_N \, \big| \, D_N \big] \big\Vert_\infty^q \; \Big| \; D_N \Big] \ge (T/N)^{q/2} \cdot \mathcal{M}_q\big( \alpha_0^{(N)}, \ldots, \alpha_{N-1}^{(N)}\big)
\end{split}
\end{equation}
holds almost surely where
\[ \alpha_\ell^{(N)} := \big\vert \sigma_N\big( t_\ell^{(N)}, \widetilde{X}_N(t_\ell^{(N)}) \big) \big\vert_{\infty,2} \]
for $\ell \in \{ 0, \ldots, N-1 \}$.

Combining \eqref{eq:serve}, \eqref{eq:topspin}, and \eqref{eq:slice2} yields
\begin{equation}
\label{eq:doublefault}
\begin{split}
&\liminf_{N \to \infty} \big( N/ \log(N) \big)^{1/2} \cdot \inf\Big\{ e_q\big( \widehat{X} \big) \; \Big| \; \widehat{X} \in \mathbb{X}_N^{\mathrm{eq}} \Big\} \\
&\ge T^{1/2} \cdot \liminf_{N \to \infty} \Big\Vert \big( \log(N) \big)^{-1/2} \cdot \mathcal{M}_q^{1/q}\big( \alpha_0^{(N)}, \ldots, \alpha_{N-1}^{(N)} \big) \Big\Vert_{L_q}.
\end{split}
\end{equation}

Next, we again use the subsequence argument that is provided by Lemma \ref{lemma:appendixB_liminf_subsequence_argument} to conclude \eqref{eq:liminf_eq} from~\eqref{eq:doublefault}. First of all, assumption~\eqref{eq:theorem_eq_convergence_in_Lq} implies
\begin{equation}
\label{eq:ace}
\alpha^{(N)} := \max_{\ell \in \{ 0, \ldots, N-1 \}} \alpha_\ell^{(N)} \quad \xrightarrow[N \to \infty]{\Prob} \quad \sup_{t \in [0,T]} \big\vert \sigma\big( t, X(t) \big) \big\vert_{\infty,2}.
\end{equation}
Now let $(\alpha^{(N_\kappa)})_{\kappa \in \N}$ be a subsequence of $(\alpha^{(N)})_{N \in \N}$. In view of \eqref{eq:ace}, there exists a subsequence $(\alpha^{(N_{\kappa_n})})_{n \in \N}$ of $(\alpha^{(N_\kappa)})_{\kappa \in \N}$ such that
\begin{equation*}
\alpha^{(N_{\kappa_n})} \quad \xrightarrow[n \to \infty]{\text{a.s.}} \quad \sup_{t \in [0,T]} \big\vert \sigma\big( t, X(t) \big) \big\vert_{\infty,2}.
\end{equation*}
Again, an easy generalization of Corollary 2 in Müller-Gronbach~\cite{tmg2002} regarding non-negative instead of strictly positive scalars leads to
\begin{equation*}
\big( \log(N_{\kappa_n}) \big)^{-1/2} \cdot \mathcal{M}_q^{1/q}\big( \alpha_0^{(N_{\kappa_n})}, \ldots, \alpha_{N_{\kappa_n}-1}^{(N_{\kappa_n})} \big) \quad \xrightarrow[n \to \infty]{\text{a.s.}} \quad 2^{-1/2} \cdot \sup_{t \in [0,T]} \big\vert \sigma\big( t, X(t) \big) \big\vert_{\infty,2}.
\end{equation*}
Consequently, Fatou's lemma gives
\begin{equation*}
\begin{split}
&\liminf_{n \to \infty} \Big\Vert \big( \log(N_{\kappa_n}) \big)^{-1/2} \cdot \mathcal{M}_q^{1/q}\big( \alpha_0^{(N_{\kappa_n})}, \ldots, \alpha_{N_{\kappa_n}-1}^{(N_{\kappa_n})} \big) \Big\Vert_{L_q} \ge 2^{-1/2} \cdot \bigg\Vert \sup_{t \in [0,T]} \big\vert \sigma\big( t, X(t) \big) \big\vert_{\infty,2} \bigg\Vert_{L_q}.
\end{split}
\end{equation*}

Finally, employing Lemma \ref{lemma:appendixB_liminf_subsequence_argument} finishes the proof of this lemma.
\end{proof}

%\newpage

\subsection{Asymptotic upper bounds}
We start by introducing some notation that will be used in this subsection. For all $q \in [1,\infty)$, for all $N \in \N$, for all $\alpha_1, \ldots, \alpha_N \in [0,\infty)$, and for all independent real-valued Brownian bridges $B_1, \ldots, B_N$ on $[0,1]$ from $0$ to $0$ we put
\[ \mathcal{G}_q( \cdot \, ; \alpha_1, \ldots, \alpha_N) : \quad [0,\infty) \to [0,1], \quad u \mapsto \Prob\bigg( \bigg\{ \max_{\ell \in \{ 1, \ldots, N \}} \Big( \alpha_\ell \cdot \sup_{t \in [0,1]} |B_\ell(t)| \Big)^q > u \bigg\} \bigg), \] 
and
\[ \mathcal{G}_q( \cdot \, ; N ) := \mathcal{G}_q( \cdot \, ; \underbrace{1, \ldots, 1}_{N \text{ times}} ). \]

First, we prove an asymptotic upper bound for the errors of the adaptive modified EM schemes.

\begin{Lemma}
\label{lemma:limsup_ad}
Assume the setting of Theorem~\ref{th:theorem_ad}. Then it holds that
\[ \limsup_{N \to \infty} \Big( c\big( \widehat{X}_{N,q}^{\mathrm{ad}} \big) / \log\big( c\big( \widehat{X}_{N,q}^{\mathrm{ad}} \big) \big) \Big)^{1/2} \cdot e_q\big( \widehat{X}_{N,q}^{\mathrm{ad}} \big) \le C_q^{\mathrm{ad}}. \]
\end{Lemma}

\begin{proof}
Fix $N \in \N$ with $1 < k_N \le N$ for the moment. Due to the triangle inequality and assumption~\eqref{eq:theorem_ad_strong_convergence}, it holds that
\begin{equation}
\label{eq:firstserve}
\begin{split}
e_q\big( \widehat{X}_{N,q}^{\mathrm{ad}} \big) &= \Big\Vert \big\Vert X - \widehat{X}_{N,q}^{\mathrm{ad}} \big\Vert_\infty \Big\Vert_{L_q} 
%\le \Big\Vert \big\Vert \widetilde{X}_N - \widehat{X}_N^{eq} \big\Vert_\infty \Big\Vert_{L_q} + \Big\Vert \big\Vert X - \widetilde{X}_N \big\Vert_\infty \Big\Vert_{L_q} \\
\le \Big\Vert \big\Vert \widetilde{X}_{k_N} - \widehat{X}_{N,q}^{\mathrm{ad}} \big\Vert_\infty \Big\Vert_{L_q} + c \cdot k_N^{-1/2}. \\
\end{split}
\end{equation}
Note that for all $\ell \in \{ 0, \ldots, k_N-1 \}$ and for all $t \in (t_\ell^{(k_N)},t_{\ell+1}^{(k_N)}]$ we have
\begin{equation*}
\begin{split}
\widetilde{X}_{k_N}(t) - \widehat{X}_{N,q}^{\mathrm{ad}}(t)
&= \sigma_{k_N}\big( t_\ell^{(k_N)}, \widetilde{X}_{k_N}(t_\ell^{(k_N)}) \big) \cdot \big( W(t) - \widehat{W}_N^{\mathrm{ad}}(t) \big)
\end{split}
\end{equation*}
where $\widehat{W}_N^{\mathrm{ad}} : \Omega \times [0,T] \to \R^m$ denotes the piecewise-linear interpolation of $W$ at the adaptive sites~\eqref{eq:adaptive_discretization}.
Recall the definitions~\eqref{eq:tildeA_kN} and \eqref{eq:eta_l} of $\A_{k_N}$ and $\eta_\ell$, respectively.
Almost identically to the proof of equation (25) in Müller-Gronbach~\cite{tmg2002}, one shows that
\begin{equation}
\label{eq:helpmeplease}
\begin{split}
&\E\left[ \big\Vert \widetilde{X}_{k_N} - \widehat{X}_{N,q}^{\mathrm{ad}} \big\Vert_\infty^q \, \middle| \, \big( \xi, W(t_1^{(k_N)}), \ldots, W(t_{k_N}^{(k_N)}) \big) \right] \\
&\le \Big( \big( \log\big( \nu_{N,q}^{\mathrm{ad}} \big) / N \big)^{1/2} \cdot 2^{-1/2} \cdot \A_{k_N}^{2/(q+2)} \Big)^q \cdot I_{\nu_{N,q}^{\mathrm{ad}}}
\end{split}
\end{equation}
holds almost surely where
\[ I_{\nu_{N,q}^{\mathrm{ad}}} := \bigg( 1 + d \cdot 2^{q/2} \cdot \int_{2^{-q/2}}^\infty \mathcal{G}_q\big(u \cdot \log( \nu_{N,q}^{\mathrm{ad}} )^{q/2}; \nu_{N,q}^{\mathrm{ad}} \big) \, \mathrm{d}u \bigg). \]
From this we conclude
\begin{equation}
\label{eq:secondserve}
\begin{split}
&\Big( c\big( \widehat{X}_{N,q}^{\mathrm{ad}} \big) / \log\big( c\big( \widehat{X}_{N,q}^{\mathrm{ad}} \big) \big) \Big)^{1/2} \cdot \Big\Vert \big\Vert \widetilde{X}_{k_N} - \widehat{X}_{N,q}^{\mathrm{ad}} \big\Vert_\infty \Big\Vert_{L_q} \\
&\le 2^{-1/2} \cdot \left( \frac{c\big( \widehat{X}_{N,q}^{\mathrm{ad}} \big)}{\log\big( c\big( \widehat{X}_{N,q}^{\mathrm{ad}} \big) \big)} \cdot \frac{\log(N)}{N} \right)^{1/2} \cdot \Big\Vert \big( \log( \nu_{N,q}^{\mathrm{ad}} ) / \log(N) \big)^{1/2} \cdot \A_{k_N}^{2/(q+2)} \cdot I_{\nu_{N,q}^{\mathrm{ad}}}^{1/q} \Big\Vert_{L_q}.
\end{split}
\end{equation}

Our main task now is to prove that the limit of the right-hand side of \eqref{eq:secondserve} is bounded above by $C_q^{\mathrm{ad}}$ as $N$ tends to infinity. To this end, note first that
\begin{equation}
\label{eq:zwischenserve1}
\A_{k_N}^{2/(q+2)} \quad \xrightarrow[N \to \infty]{L_q} \quad \bigg( \int_0^T \big\vert \sigma\big( t, X(t) \big) \big\vert_{\infty,2}^2 \, \mathrm{d}t \bigg)^{1/(q+2)}
\end{equation}
holds due to assumption \eqref{eq:theorem_ad_convergence_in_Lq}.
We next separately analyze the asymptotics of the two relevant terms appearing in the right-hand side of \eqref{eq:secondserve}. First, straightforward calculations using the estimates~\eqref{eq:inequality_totalNumber_nachOben} and \eqref{eq:inequality_totalNumber_nachUnten} along with the limits \eqref{eq:limits_kN} and \eqref{eq:zwischenserve1} show
\begin{equation}
\label{eq:zwischenserve2}
\begin{split}
\lim_{N \to \infty} \frac{c\big( \widehat{X}_{N,q}^{\mathrm{ad}} \big)}{\log\big( c\big( \widehat{X}_{N,q}^{\mathrm{ad}} \big) \big)} \cdot \frac{\log(N)}{N}
&= \lim_{N \to \infty} \frac{c\big( \widehat{X}_{N,q}^{\mathrm{ad}} \big) / N}{1 + \log\big( c\big( \widehat{X}_{N,q}^{\mathrm{ad}} \big) / N \big) / \log(N)} \\
&= \bigg\Vert \bigg( \int_0^T \big\vert \sigma\big( t, X(t) \big) \big\vert_{\infty,2}^2 \, \mathrm{d}t \bigg)^{1/2} \bigg\Vert_{L_{2q/(q+2)}}^{2q/(q+2)}.
\end{split}
\end{equation}
Second, observe that \eqref{eq:inequality_totalNumber_nachOben}, the inequality $\log(1+x) \le x$ for all $x \in (-1,\infty)$, the inequality $\sqrt{1+x} \le 1+\sqrt{x}$ for all $x \in [0,\infty)$, and the triangle inequality yield
\begin{equation}
\label{eq:helpX}
\begin{split}
&\Big\Vert \big( \log( \nu_{N,q}^{\mathrm{ad}} ) / \log(N) \big)^{1/2} \cdot \A_{k_N}^{2/(q+2)} \cdot I_{\nu_{N,q}^{\mathrm{ad}}}^{1/q} \Big\Vert_{L_q} \\
&\le \Big\Vert \big(1+ \A_{k_N}^{2q/(q+2)} / \log(N) \big)^{1/2} \cdot \A_{k_N}^{2/(q+2)} \cdot I_{\nu_{N,q}^{\mathrm{ad}}}^{1/q} \Big\Vert_{L_q} \\
&\le \Big\Vert \A_{k_N}^{2/(q+2)} \cdot I_{\nu_{N,q}^{\mathrm{ad}}}^{1/q} \Big\Vert_{L_q} + \Big\Vert \A_{k_N} \cdot I_{\nu_{N,q}^{\mathrm{ad}}}^{1/q}  \cdot \big( \log(N) \big)^{-1/2} \Big\Vert_{L_q} \\
\end{split}
\end{equation}
for all $N \in \N$ with $1 < k_N \le N$. Furthermore, note that
$\nu_{N,q}^{\mathrm{ad}}$ tends to infinity as $N$ tends to infinity due to \eqref{eq:inequality_totalNumber_nachUnten}. Hence, Lemma 2 in Müller-Gronbach~\cite{tmg2002} implies
\begin{equation}
\label{eq:fourthserve}
I_{\nu_{N,q}^{\mathrm{ad}}} \quad \xrightarrow[N \to \infty]{\text{a.s.}} \quad 1
\end{equation}
and that
\begin{equation}
\label{eq:fourthserve2}
\sup_{N \in \N} I_{\nu_{N,q}^{\mathrm{ad}}} \le c
\end{equation}
holds almost surely.
Combining \eqref{eq:zwischenserve1}, \eqref{eq:fourthserve}, and \eqref{eq:fourthserve2} gives
\begin{equation}
\label{eq:fifthserve}
\A_{k_N}^{2/(q+2)} \cdot I_{\nu_{N,q}^{\mathrm{ad}}}^{1/q} \quad \xrightarrow[N \to \infty]{L_q} \quad \bigg( \int_0^T \big\vert \sigma\big( t, X(t) \big) \big\vert_{\infty,2}^2 \, \mathrm{d}t \bigg)^{1/(q+2)}
\end{equation}
and
\begin{equation}
\label{eq:fifthserve2}
\A_{k_N} \cdot I_{\nu_{N,q}^{\mathrm{ad}}}^{1/q} \cdot \big( \log(N) \big)^{-1/2} \quad \xrightarrow[N \to \infty]{L_q} \quad 0.
\end{equation}

Finally, combining \eqref{eq:firstserve}, \eqref{eq:secondserve}, \eqref{eq:zwischenserve2}, \eqref{eq:helpX}, \eqref{eq:fifthserve}, \eqref{eq:fifthserve2}, and \eqref{eq:limits_kN} finishes the proof of this lemma.
\end{proof}

%\newpage

Next, we prove an asymptotic upper bound for the errors of the equidistant modified EM schemes.

\begin{Lemma}
\label{lemma:limsup_eq}
Assume the setting of Theorem~\ref{th:theorem_eq}. Then it holds that
\[ \limsup_{N \to \infty} \big( N / \log(N) \big)^{1/2} \cdot e_q\big( \widehat{X}_N^{\mathrm{eq}} \big) \le C_q^{\mathrm{eq}}. \]
\end{Lemma}

\begin{proof}
Fix $N \in \N$ with $N > 1$ for the moment. Similarly to the estimates~\eqref{eq:firstserve} and \eqref{eq:helpmeplease} in the proof of Lemma~\ref{lemma:limsup_ad}, one successively shows
\begin{equation}
\label{eq:test1}
\begin{split}
e_q\big( \widehat{X}_N^{\mathrm{eq}} \big) &= \Big\Vert \big\Vert X - \widehat{X}_N^{\mathrm{eq}} \big\Vert_\infty \Big\Vert_{L_q} 
%\le \Big\Vert \big\Vert \widetilde{X}_N - \widehat{X}_N^{eq} \big\Vert_\infty \Big\Vert_{L_q} + \Big\Vert \big\Vert X - \widetilde{X}_N \big\Vert_\infty \Big\Vert_{L_q} \\
\le \Big\Vert \big\Vert \widetilde{X}_N - \widehat{X}_N^{\mathrm{eq}} \big\Vert_\infty \Big\Vert_{L_q} + c \cdot N^{-1/2} \\
\end{split}
\end{equation}
and that
%\begin{equation*}
%\begin{split}
%\widetilde{X}_N(t) - \widehat{X}_N^{\mathrm{eq}}(t)
%&= \sigma_N\big( t_\ell^{(N)}, \widetilde{X}_N(t_\ell^{(N)}) \big) \cdot \big( W(t) - \widehat{W}_N^{\mathrm{eq}}(t) \big)
%%&= \frac{\sigma\big( t_\ell^{(N)}, \widetilde{X}_N(t_\ell^{(N)}) \big)}{1+(T/N)^{1/2} \cdot \big| \widetilde{X}_N(t_\ell^{(N)}) \big|^r} \cdot \Big( W(t) - \widehat{W}_N^{\mathrm{eq}}(t) \Big)
%\end{split}
%\end{equation*}
%holds for all $\ell \in \{ 0, \ldots, N-1 \}$ and for all $t \in (t_\ell^{(N)},t_{\ell+1}^{(N)}]$ where $\widehat{W}_N^{\mathrm{eq}} : \Omega \times [0,T] \to \R^m$ denotes the piecewise-linear interpolation of $W$ at the equidistant sites \eqref{eq:equidistant_discretization}. Almost identically to the proof of equation (25) in Müller-Gronbach~\cite{tmg2002}, one shows that
\begin{equation*}
\begin{split}
&\E\left[ \big\Vert \widetilde{X}_N - \widehat{X}_N^{\mathrm{eq}} \big\Vert_\infty^q \; \middle| \; \big( \xi, W(t_1^{(N)}), \ldots, W(t_N^{(N)}) \big) \right] \\
&\le \bigg( (T/2)^{1/2} \cdot \big( \log(N)/N \big)^{1/2} \cdot \max_{\ell \in \{ 0, \ldots, N-1 \}} \big\vert \sigma_N\big( t_\ell^{(N)}, \widetilde{X}_N(t_\ell^{(N)}) \big) \big\vert_{\infty,2} \bigg)^q \cdot I_N
%&\le \Bigg( (T/2)^{1/2} \cdot \big( \log(N)/N \big)^{1/2} \cdot \max_{\ell \in \{ 0, \ldots, N-1 \}} \Bigg\vert \frac{\sigma\big( t_\ell^{(N)}, \widetilde{X}_N(t_\ell^{(N)}) \big)}{1+(T/N)^{1/2} \cdot \big| \widetilde{X}_N(t_\ell^{(N)}) \big|^r} \Bigg\vert_{\infty,2} \Bigg)^q \cdot I_N
\end{split}
\end{equation*}
holds almost surely where
\[ I_N := \bigg( 1 + d \cdot 2^{q/2} \cdot \int_{2^{-q/2}}^\infty \mathcal{G}_q\big(u \cdot \log(N)^{q/2}; N\big) \, \mathrm{d}u \bigg). \]
Thus, we conclude that
\begin{equation}
\label{eq:limsup}
\begin{split}
&\big( N / \log(N) \big)^{1/2} \cdot \Big\Vert \big\Vert \widetilde{X}_N - \widehat{X}_N^{\mathrm{eq}} \big\Vert_\infty \Big\Vert_{L_q} \\
&\le (T/2)^{1/2} \cdot \bigg\Vert \max_{\ell \in \{ 0, \ldots, N-1 \}} \big\vert \sigma_N\big( t_\ell^{(N)}, \widetilde{X}_N(t_\ell^{(N)}) \big) \big\vert_{\infty,2} \bigg\Vert_{L_q} \cdot I_N^{1/q}.
\end{split}
\end{equation}

As a final step, we show that the right-hand side of \eqref{eq:limsup} tends to $C_q^\mathrm{eq}$ as $N$ tends to infinity. To this end, note first that
\begin{equation}
\label{eq:test3}
\max_{\ell \in \{ 0, \ldots, N-1 \}} \big\vert \sigma_N\big( t_\ell^{(N)}, \widetilde{X}_N(t_\ell^{(N)}) \big) \big\vert_{\infty,2} \quad \xrightarrow[N \to \infty]{L_q} \quad \sup_{t \in [0,T]} \big\vert \sigma\big( t, X(t) \big) \big\vert_{\infty,2}
\end{equation}
holds due to assumption~\eqref{eq:theorem_eq_convergence_in_Lq}.
Moreover, Lemma~2 in Müller-Gronbach~\cite{tmg2002} gives
\begin{equation}
\label{eq:test2}
\lim_{N \to \infty} I_N = 1.
\end{equation}

Finally, combining \eqref{eq:test1}, \eqref{eq:limsup}, \eqref{eq:test3}, and \eqref{eq:test2} finishes the proof of this lemma.
%Finally, combining \eqref{eq:test1}, \eqref{eq:limsup}, \eqref{eq:test2} and \eqref{eq:test3} yields
%\begin{equation*}
%\begin{split}
%&\limsup_{N \to \infty} \left( \frac{N}{\log(N)} \right)^{1/2} \cdot e_q\big( \widehat{X}_N^{\mathrm{eq}} \big) \\
%&\le \limsup_{N \to \infty} \sqrt{\frac{T}{2}} \cdot \left\Vert \max_{l \in \{ 0, \ldots, N-1 \}} \vertiii{\frac{\sigma\big( t_\ell^{(N)}, \widetilde{X}_N(t_\ell^{(N)}) \big)}{1+(T/N)^{1/2} \cdot \big| \widetilde{X}_N(t_\ell^{(N)}) \big|^r}} \right\Vert_{L_q} \cdot I_N^{1/q} = C_q^{\mathrm{eq}}
%\end{split}
%\end{equation*}
%which completes the proof of this lemma.
\end{proof}

%\newpage
%%%%%%%%%%%%%%%%%%%%%%%%%%%%%%%%%%%%%%%%%%%%%%%%%%%%%%
\section{Future Work}
\label{sec:FutureWork}
%%%%%%%%%%%%%%%%%%%%%%%%%%%%%%%%%%%%%%%%%%%%%%%%%%%%%%

Throughout this paper, we studied strongly asymptotically optimal approximations with respect to the particular $q$th mean supremum error \eqref{eq:error_criterion}. Besides, the $q$th mean $L_q$ distance of an approximation $\widehat{X}$, given by
\[ \widetilde{e}_q\big( \widehat{X} \big) := \bigg( \E\bigg[ \int_0^T \sum_{i=1}^d \big| X_i(t) - \widehat{X}_i(t) \big|^q \, \mathrm{d}t \bigg] \bigg)^{1/q} \]
for $q \in [1,\infty)$, is another error measure commonly analyzed in the literature. For SDEs whose coefficients as well as their partial derivatives are globally Lipschitz continuous, Müller-Gronbach~\cite{tmgHabil} showed that specific Milstein schemes relating to adaptive and to equidistant time discretizations perform strongly asymptotically optimal in the classes of adaptive and of equidistant approximations, respectively. To generalize these results to a wider class of SDEs, it appears very promising to switch from Milstein schemes to suitable coefficient-modified Milstein schemes. 
%The assumptions in the thereby obtained results on strong asymptotic optimality in the classes of adaptive and equidistant have to be altered accordingly. 
The classical Milstein schemes as well as certain tamed Milstein schemes (similarly to the ones defined in Gan and Wang~\cite{ganwang2013} or Kumar and Sabanis~\cite{sabanistamedMilstein}) might then represent two exemplary applications of such new results. The whole approach described above may constitute the object of future studies.

%\newpage
%%%%%%%%%%%%%%%%%%%%%%%%%%%%%%%%%%%%%%%%%%%%%%%%%%%%%%
\appendix
\section{Properties of the Solution Process and of
the Continuous-time Tamed Euler Schemes}
\label{sec:Appendix}
%%%%%%%%%%%%%%%%%%%%%%%%%%%%%%%%%%%%%%%%%%%%%%%%%%%%%%

In this appendix, we provide useful properties of the solution process $(X(t))_{t \in [0,T]}$ and the continuous-time tamed Euler schemes $(\widetilde{X}_N(t))_{t \in [0,T]}$, $N \in \N$, with $(\mu_N)_{N \in \N}$ and $(\sigma_N)_{N \in \N}$ as per~\eqref{eq:muN_sigmaN_tamedEuler}. More precisely, we prove boundedness of certain moments of the suprema of these processes as well as strong convergence of order $1/2$ for the continuous-time tamed Euler schemes.

As before, we use $c$ to denote unspecified positive constants that may vary at every occurrence and that may only depend on $T$, $d$, $m$, and the parameters and constants from the assumptions used in the respective propositions.

First, we consider the supremum of the solution of the SDE \eqref{eq:SDE} and prove finiteness of specific moments of this random variable under quite weak assumptions.
%We initially show finiteness of certain moments of both the solution of the SDE \eqref{eq:SDE} and the continuous-time tamed Euler schemes.
%These properties are frequently exploited in the proofs of Lemma \ref{lemma:distance_sigma_sigma_n_in_prob} and of the asymptotic upper bounds.

%\subsection{On Some Properties of the Solution and the Continuous-time Tamed Euler Method}

\begin{Proposition}
\label{prop:EsupX}
Let the Assumptions \hyperlink{ass:I}{\normalfont (I$_{p}$)}, \hyperlink{ass:locL}{\normalfont (locL)}, \hyperlink{ass:K}{\normalfont (K$_{p}$)}, and \hyperlink{ass:pG}{\normalfont (pG$_r^{\sigma}$)} be satisfied for some $p \in [2,\infty)$ and $r \in [1,\infty)$ with $p \ge 2r$. Then it holds that
\begin{equation*}
%\label{eq:goal}
\E\bigg[ \sup_{t \in [0,T]} \big| X(t) \big|^{p-2r+2} \bigg] < \infty.
\end{equation*}
\end{Proposition}

\begin{proof}
Put $\overline{p} := p - 2r + 2 \in [2,p]$. For each $n \in \N$, observe that the mapping
\[ \tau_n : \quad \Omega \to [0,T], \quad \omega \mapsto T \wedge \inf\big\{ t \in [0,T] \; \big| \; n \le | X(t,\omega) | \big\}, \]
is a stopping time that satisfies
\begin{equation}
\label{eq:finitenessOfStoppedSolution}
\sup_{t \in [0,T]} \big| X(t \wedge \tau_n) \big| \le \max\big\{ n, |\xi| \big\}
\end{equation}
almost surely.

Fix $n \in \N$ for the moment. Employing Itô's formula and Assumption~\hyperlink{ass:K}{\normalfont (K$_{p}$)} yields that almost surely we have
\begin{equation*}
\begin{split}
&\left( 1 + \big| X(t \wedge \tau_n) \big|^2 \right)^{\overline{p}/2} \\
%&\le \left( 1 + \big| \xi \big|^2 \right)^{\overline{p}/2} + \frac{\overline{p} \cdot C}{2} \cdot \int_0^{t \wedge \tau_n} \left( 1 + \big| X(s) \big|^2 \right)^{\overline{p}/2} \, \mathrm{d}s \\
%&\hspace{0.4cm}+ \overline{p} \cdot \int_0^{t \wedge \tau_n} \left( 1 + \big| X(s) \big|^2 \right)^{(\overline{p}-2)/2} \cdot X(s)^\top \cdot \sigma\big( s, X(s) \big) \, \mathrm{d}W(s) \\
%&= \left( 1 + \big| \xi \big|^2 \right)^{\overline{p}/2} + \frac{\overline{p} \cdot C}{2} \cdot \int_0^t \mathds{1}_{\{ s \le \tau_n \}} \cdot \left( 1 + \big| X(s) \big|^2 \right)^{\overline{p}/2} \, \mathrm{d}s \\
%&\hspace{0.4cm}+ \overline{p} \cdot \int_0^t \mathds{1}_{\{ s \le \tau_n \}} \cdot \left( 1 + \big| X(s) \big|^2 \right)^{(\overline{p}-2)/2} \cdot X(s)^\top \cdot \sigma\big( s, X(s) \big) \, \mathrm{d}W(s) \\
%&= \left( 1 + \big| \xi \big|^2 \right)^{\overline{p}/2} + \frac{\overline{p} \cdot C}{2} \cdot \int_0^t \mathds{1}_{\{ s \le \tau_n \}} \cdot \left( 1 + \big| X(s \wedge \tau_n) \big|^2 \right)^{\overline{p}/2} \, \mathrm{d}s \\
%&\hspace{0.4cm}+ \overline{p} \cdot \int_0^t \mathds{1}_{\{ s \le \tau_n \}} \cdot \left( 1 + \big| X(s \wedge \tau_n) \big|^2 \right)^{(\overline{p}-2)/2} \cdot X(s \wedge \tau_n)^\top \cdot \sigma\big( s \wedge \tau_n, X(s \wedge \tau_n) \big) \, \mathrm{d}W(s) \\
&\le \left( 1 + \big| \xi \big|^2 \right)^{\overline{p}/2} + c \cdot \int_0^t \left( 1 + \big| X(s \wedge \tau_n) \big|^2 \right)^{\overline{p}/2} \, \mathrm{d}s \\
&\hspace{0.4cm}+ \overline{p} \cdot \int_0^t \mathds{1}_{\{ s \le \tau_n \}} \cdot \left( 1 + \big| X(s \wedge \tau_n) \big|^2 \right)^{(\overline{p}-2)/2} \cdot X(s \wedge \tau_n)^\top \cdot \sigma\big( s \wedge \tau_n, X(s \wedge \tau_n) \big) \, \mathrm{d}W(s)
\end{split}
\end{equation*}
for all $t \in [0,T]$. Thus, Assumption \hyperlink{ass:I}{\normalfont (I$_{p}$)}, Fubini's theorem, and the moments estimate \eqref{eq:supEX_finite} give
%\begin{equation*}
%\begin{split}
%&\sup_{u \in [0,t]} \left( 1 + \big| X(u \wedge \tau_n) \big|^2 \right)^{\overline{p}/2} \\
%&\le \left( 1 + \big| \xi \big|^2 \right)^{\overline{p}/2} + \frac{\overline{p} \cdot C}{2} \cdot \sup_{u \in [0,t]} \int_0^u \left( 1 + \big| X(s \wedge \tau_n) \big|^2 \right)^{\overline{p}/2} \, \mathrm{d}s \\
%&\hspace{0.4cm}+ \overline{p} \cdot \sup_{u \in [0,t]} \int_0^u \mathds{1}_{\{ s \le \tau_n \}} \cdot \left( 1 + \big| X(s \wedge \tau_n) \big|^2 \right)^{(\overline{p}-2)/2} \cdot X(s \wedge \tau_n)^\top \cdot \sigma\big( s \wedge \tau_n, X(s \wedge \tau_n) \big) \, \mathrm{d}W(s) \\
%&= \left( 1 + \big| \xi \big|^2 \right)^{\overline{p}/2} + \frac{\overline{p} \cdot C}{2} \cdot \int_0^t \left( 1 + \big| X(s \wedge \tau_n) \big|^2 \right)^{\overline{p}/2} \, \mathrm{d}s \\
%&\hspace{0.4cm}+ \overline{p} \cdot \sup_{u \in [0,t]} \int_0^u \mathds{1}_{\{ s \le \tau_n \}} \cdot \left( 1 + \big| X(s \wedge \tau_n) \big|^2 \right)^{(\overline{p}-2)/2} \cdot X(s \wedge \tau_n)^\top \cdot \sigma\big( s \wedge \tau_n, X(s \wedge \tau_n) \big) \, \mathrm{d}W(s)
%\end{split}
%\end{equation*}
%for all $t \in [0,T]$ almost surely and consequently
\begin{equation}
\label{eq:woohoo}
\begin{split}
&\E\bigg[ \sup_{t \in [0,T]} \left( 1 + \big| X(t \wedge \tau_n) \big|^2 \right)^{\overline{p}/2} \bigg] \\
%&\le \E\bigg[ \left( 1 + \big| \xi \big|^2 \right)^{\overline{p}/2} \bigg] + \frac{\overline{p} \cdot C}{2} \cdot \E\bigg[ \int_0^T \left( 1 + \big| X(s \wedge \tau_n) \big|^2 \right)^{\overline{p}/2} \, \mathrm{d}s \bigg] \\
%&\hspace{0.4cm}+ \overline{p} \cdot \E\bigg[ \sup_{t \in [0,T]} \int_0^t \mathds{1}_{\{ s \le \tau_n \}} \cdot \left( 1 + \big| X(s \wedge \tau_n) \big|^2 \right)^{(\overline{p}-2)/2} \\
%&\hspace{3.175cm}\cdot X(s \wedge \tau_n)^\top \cdot \sigma\big( s \wedge \tau_n, X(s \wedge \tau_n) \big) \, \mathrm{d}W(s) \bigg] \\
&\le c + \overline{p} \cdot \E\bigg[ \sup_{t \in [0,T]} \int_0^t \mathds{1}_{\{ s \le \tau_n \}} \cdot \left( 1 + \big| X(s \wedge \tau_n) \big|^2 \right)^{(\overline{p}-2)/2} \\
&\hspace{3.395cm}\cdot X(s \wedge \tau_n)^\top \cdot \sigma\big( s \wedge \tau_n, X(s \wedge \tau_n) \big) \, \mathrm{d}W(s) \bigg].
\end{split}
\end{equation}
Next, observe that the Burkholder--Davis--Gundy inequality and the Cauchy--Schwarz inequality imply
\begin{equation}
\label{eq:hattrick}
\begin{split}
&\E\bigg[ \sup_{t \in [0,T]} \int_0^t \mathds{1}_{\{ s \le \tau_n \}} \cdot \left( 1 + \big| X(s \wedge \tau_n) \big|^2 \right)^{(\overline{p}-2)/2} \\
&\hspace{1.92cm}\cdot X(s \wedge \tau_n)^\top \cdot \sigma\big( s \wedge \tau_n, X(s \wedge \tau_n) \big) \, \mathrm{d}W(s)\bigg] \\
&\le 32^{1/2} \cdot \E\bigg[ \bigg( \int_0^T \mathds{1}_{\{ s \le \tau_n \}} \cdot \left( 1 + \big| X(s \wedge \tau_n) \big|^2 \right)^{\overline{p}-1} \cdot \big| \sigma\big( s \wedge \tau_n, X(s \wedge \tau_n) \big) \big|^2 \, \mathrm{d}s \bigg)^{1/2} \bigg]. \\
\end{split}
\end{equation}
Moreover, Assumption \hyperlink{ass:pG}{\normalfont (pG$_r^\sigma$)} and the inequality $\sqrt{x \cdot y} \le x/(2\rho) + y\rho/2$ for all $x,y \in [0,\infty)$ and $\rho \in (0,\infty)$ yield
\begin{equation}
\label{eq:hattrick2}
\begin{split}
&\E\bigg[ \bigg( \int_0^T \mathds{1}_{\{ s \le \tau_n \}} \cdot \left( 1 + \big| X(s \wedge \tau_n) \big|^2 \right)^{\overline{p}-1} \cdot \big| \sigma\big( s \wedge \tau_n, X(s \wedge \tau_n) \big) \big|^2 \, \mathrm{d}s \bigg)^{1/2} \bigg] \\
%&\le c \cdot \E\bigg[ \bigg( \int_0^T \mathds{1}_{\{ s \le \tau_n \}} \cdot \left( 1 + \big| X(s \wedge \tau_n) \big|^2 \right)^{\overline{p}-1} \cdot \left( 1 + \big| X(s \wedge \tau_n) \big|^{2r} \right) \, \mathrm{d}s \bigg)^{1/2} \bigg] \\
&\le c \cdot \E\bigg[ \bigg( \sup_{t \in [0,T]} \left( 1 + \big| X(t \wedge \tau_n) \big|^2 \right)^{\overline{p}/2} \\
&\hspace{1.7cm}\cdot \int_0^T \mathds{1}_{\{ s \le \tau_n \}} \cdot \left( 1 + \big| X(s \wedge \tau_n) \big|^2 \right)^{\overline{p}/2-1} \cdot \left( 1 + \big| X(s \wedge \tau_n) \big|^{2r} \right) \, \mathrm{d}s \bigg)^{1/2} \bigg] \\
%&\le c \cdot \E\left[ \left( \sup_{t \in [0,T]} \left( 1 + \big| X(t \wedge \tau_n) \big|^2 \right)^{\overline{p}/2} \cdot \int_0^t \mathds{1}_{\{ s \le \tau_n \}} \cdot \left( 1 + \big| X(s \wedge \tau_n) \big|^2 \right)^{\overline{p}/2-1+(r+2)/2} \, \mathrm{d}s \right)^{1/2} \right] \\
&\le c \cdot \E\bigg[ \bigg( \sup_{t \in [0,T]} \left( 1 + \big| X(t \wedge \tau_n) \big|^2 \right)^{\overline{p}/2} \cdot \int_0^T \mathds{1}_{\{ s \le \tau_n \}} \cdot \left( 1 + \big| X(s \wedge \tau_n) \big|^2 \right)^{p/2} \, \mathrm{d}s \bigg)^{1/2} \bigg] \\
&\le \frac{1}{2 \cdot 32^{1/2} \cdot \overline{p}} \cdot \E\bigg[ \sup_{t \in [0,T]} \left( 1 + \big| X(t \wedge \tau_n) \big|^2 \right)^{\overline{p}/2} \bigg] \\
&\hspace{0.4cm}+ c^2 \cdot 32^{1/2} \cdot \overline{p}/2 \cdot \E\bigg[ \int_0^T \mathds{1}_{\{ s \le \tau_n \}} \cdot \left( 1 + \big| X(s \wedge \tau_n) \big|^2 \right)^{p/2} \, \mathrm{d}s \bigg].
%&= \frac{1}{2 \cdot \overline{p}} \cdot \E\left[ \sup_{t \in [0,T]} \left( 1 + \big| X(t \wedge \tau_n) \big|^2 \right)^{\overline{p}/2} \right] + c \cdot \E\left[ \int_0^T \mathds{1}_{\{ s \le \tau_n \}} \cdot \left( 1 + \big| X(s \wedge \tau_n) \big|^2 \right)^{p/2} \, \mathrm{d}s \right] \\
%&\le \frac{1}{2 \cdot \overline{p}} \cdot \E\left[ \sup_{t \in [0,T]} \left( 1 + \big| X(t \wedge \tau_n) \big|^2 \right)^{\overline{p}/2} \right] + c. \\
\end{split}
\end{equation}
Note that
\begin{equation}
\label{eq:hattrick3}
\E\bigg[ \int_0^T \mathds{1}_{\{ s \le \tau_n \}} \cdot \left( 1 + \big| X(s \wedge \tau_n) \big|^2 \right)^{p/2} \, \mathrm{d}s \bigg] \le \E\bigg[ \int_0^T \left( 1 + \big| X(s) \big|^2 \right)^{p/2} \, \mathrm{d}s \bigg] \le c
\end{equation}
holds, again, due to Fubini's theorem and \eqref{eq:supEX_finite}.
%Weiter gilt für jedes $t \in [0,T]$ mit der Hölder-Ungleichung im Fall $p > 2$, dem Satz von Fubini sowie Satz 1(ii)
%\begin{equation*}
%\begin{split}
%\E\left[ \int_0^t \mathds{1}_{\{ s \le \tau_n \}} \cdot \left( 1 + \big| X(s \wedge \tau_n) \big|^2 \right)^{p/2} \, \mathrm{d}s \right] &\le \E\left[ \int_0^t \mathds{1}_{\{ s \le \tau_n \}} \cdot \left( 1 + \big| X(s \wedge \tau_n) \big|^{p} \right) \cdot 2^{(p-2)/2} \, \mathrm{d}s \right] \\
%&= 2^{(p-2)/2} \cdot \E\left[ \int_0^t \mathds{1}_{\{ s \le \tau_n \}} \cdot \left( 1 + \big| X(s) \big|^{p} \right) \, \mathrm{d}s \right] \\
%&\le 2^{(p-2)/2} \cdot \E\left[ \int_0^t \left( 1 + \big| X(s) \big|^{p} \right) \, \mathrm{d}s \right] \\
%&= 2^{(p-2)/2} \cdot \int_0^t \left( 1 + \E\left[ \big| X(s) \big|^{p} \right] \right) \, \mathrm{d}s \\
%&\le 2^{(p-2)/2} \cdot T \cdot \left( 1 + \sup_{t \in [0,T]} \E\left[ \big| X(u) \big|^{p} \right] \right) \\
%&\le 2^{(p-2)/2} \cdot T \cdot \left( 1 + c \right).
%\end{split}
%\end{equation*}
Combining the inequalities \eqref{eq:woohoo}, \eqref{eq:hattrick}, \eqref{eq:hattrick2}, and \eqref{eq:hattrick3} shows
\begin{equation*}
\begin{split}
\E\bigg[ \sup_{t \in [0,T]} \left( 1 + \big| X(t \wedge \tau_n) \big|^2 \right)^{\overline{p}/2} \bigg] &\le \frac{1}{2} \cdot \E\bigg[ \sup_{t \in [0,T]} \left( 1 + \big| X(t \wedge \tau_n) \big|^2 \right)^{\overline{p}/2} \bigg] + c.
%\E\left[ \sup_{t \in [0,T]} \left( 1 + \big| X(u \wedge \tau_n) \big|^2 \right)^{\overline{p}/2} \right] &\le 2^{(p-2)/2} \cdot \left( 1 + \E\left[ \big| \xi \big|^{p} \right] \right) + \frac{\overline{p} \cdot C}{2} \cdot \int_0^t \E\left[ \sup_{u \in [0,s]} \left( 1 + \big| X(u \wedge \tau_n) \big|^2 \right)^{\overline{p}/2} \right] \, \mathrm{d}s \\
%&\hspace{0.4cm}+ \overline{p} \cdot \left\{ \frac{1}{2 \cdot \overline{p}} \cdot \E\left[ \sup_{t \in [0,T]} \left( 1 + \big| X(u \wedge \tau_n) \big|^2 \right)^{\overline{p}/2} \right] + \overline{p} \cdot 64 \cdot C^2 \cdot 2^{(p-2)/2} \cdot T \cdot \left( 1 + c \right) \right\} \\
%&= 2^{(p-2)/2} \cdot \left( 1 + \E\left[ \big| \xi \big|^{p} \right] \right) + \frac{\overline{p} \cdot C}{2} \cdot \int_0^t \E\left[ \sup_{u \in [0,s]} \left( 1 + \big| X(u \wedge \tau_n) \big|^2 \right)^{\overline{p}/2} \right] \, \mathrm{d}s \\
%&\hspace{0.4cm}+ \frac{1}{2} \cdot \E\left[ \sup_{t \in [0,T]} \left( 1 + \big| X(u \wedge \tau_n) \big|^2 \right)^{\overline{p}/2} \right] + \left( \overline{p} \cdot C \right)^2 \cdot 64 \cdot 2^{(p-2)/2} \cdot T \cdot \left( 1 + c \right)
\end{split}
\end{equation*}
To subtract the first summand of the right-hand side from the left-hand side, we need to ensure that these quantities are actually not infinite. For this purpose, we employ \eqref{eq:finitenessOfStoppedSolution} and Assumption~\hyperlink{ass:I}{\normalfont (I$_{p}$)} to conclude that
\begin{equation*}
%\label{eq:defender}
\E\bigg[ \sup_{t \in [0,T]} \left( 1 + \big| X(t \wedge \tau_n) \big|^2 \right)^{\overline{p}/2} \bigg] \le \E\bigg[ \left( 1+\max\big\{ n, |\xi| \big\}^2 \right)^{\overline{p}/2} \bigg] < \infty.
\end{equation*}
Hence, we obtain
\begin{equation}
\label{eq:halftime}
\E\bigg[ \sup_{t \in [0,T]} \big| X(t \wedge \tau_n) \big|^{\overline{p}} \bigg] \le \E\bigg[ \sup_{t \in [0,T]} \left( 1 + \big| X(t \wedge \tau_n) \big|^2 \right)^{\overline{p}/2} \bigg] \le c.
\end{equation}
%Das Lemma von Gronwall liefert schließlich für jedes $t \in [0,T]$
%\begin{equation*}
%\begin{split}
%&\E\left[ \sup_{t \in [0,T]} \left( 1 + \big| X(u \wedge \tau_n) \big|^2 \right)^{\overline{p}/2} \right] \le \left( 2^{p/2} \cdot \left( 1 + \E\left[ \big| \xi \big|^{p} \right] \right) + \left( \overline{p} \cdot C \right)^2 \cdot 64 \cdot 2^{p/2} \cdot T \cdot \left( 1 + c \right) \right) \cdot e^{\overline{p} \cdot C \cdot t}
%\end{split}
%\end{equation*}
%und somit
%\begin{equation}
%\label{eq:inequality_Estoppingtimes}
%\begin{split}
%\E\left[ \sup_{t \in [0,T]} \left| X(u \wedge \tau_n) \right|^{\overline{p}} \right] &\le \E\left[ \sup_{t \in [0,T]} \left( 1 + \big| X(u \wedge \tau_n) \big|^2 \right)^{\overline{p}/2} \right] \\
%&\le \left( 2^{p/2} \cdot \left( 1 + \E\left[ \big| \xi \big|^{p} \right] \right) + \left( \overline{p} \cdot C \right)^2 \cdot 64 \cdot 2^{p/2} \cdot T \cdot \left( 1 + c \right) \right) \cdot e^{\overline{p} \cdot C \cdot T}.
%\end{split}
%\end{equation}
%Weiter gilt
%\begin{equation}
%\label{eq:lim_stopping_time_solution}
%\lim_{n \to \infty} \sup_{t \in [0,T]} \big| X(u \wedge \tau_n) \big|^{\overline{p}} = \sup_{t \in [0,T]} \big| X(u) \big|^{\overline{p}}
%\end{equation}
%$\Prob$-fast sicher. \\

Using Fatou's lemma, we derive from \eqref{eq:halftime} that
\begin{equation*}
\begin{split}
\E\bigg[ \sup_{t \in [0,T]} \big| X(t) \big|^{\overline{p}} \bigg] &= \E\bigg[ \lim_{n \to \infty} \sup_{t \in [0,T]} \big| X(t \wedge \tau_n) \big|^{\overline{p}} \bigg]
%&= \E\left[ \liminf_{n \to \infty} \sup_{t \in [0,T]} \left| X(u \wedge \tau_n) \right|^{\overline{p}} \right] \\
\le \liminf_{n \to \infty} \E\bigg[ \sup_{t \in [0,T]} \big| X(t \wedge \tau_n) \big|^{\overline{p}} \bigg]
%&\le \liminf_{n \to \infty} \left( 2^{p/2} \cdot \left( 1 + \E\left[ \big| \xi \big|^{p} \right] \right) + \left( \overline{p} \cdot C \right)^2 \cdot 64 \cdot 2^{p/2} \cdot T \cdot \left( 1 + c \right) \right) \cdot e^{\overline{p} \cdot C \cdot T}
%&= \left( 2^{p/2} \cdot \left( 1 + \E\left[ \big| \xi \big|^{p} \right] \right) + \left( p-r \right)^2 \cdot C^2 \cdot 64 \cdot 2^{p/2} \cdot T \cdot \left( 1 + c \right) \right) \cdot e^{\overline{p} \cdot C \cdot T} \\
\le c,
\end{split}
\end{equation*}
which finishes the proof of this proposition.
\end{proof}

Next, we show an analogous result on moment bounds for the continuous-time tamed Euler schemes.

\begin{Proposition}
\label{prop:EsupTildeX}
Let the Assumptions \hyperlink{ass:I}{\normalfont (I$_{p}$)}, \hyperlink{ass:locL}{\normalfont (locL)}, \hyperlink{ass:K}{\normalfont (K$_{p}$)}, and \hyperlink{ass:pG}{\normalfont (pG$_{r}^{\mu}$)} be satisfied for some $p \in [2,\infty)$ and $r \in [1,\infty)$ with $p \ge r+1$. Moreover, let the functions $(\mu_N)_{N \in \N}$ and $(\sigma_N)_{N \in \N}$ be given by \eqref{eq:muN_sigmaN_tamedEuler}. Then it holds that
\begin{equation}
\label{eq:supEsupTildeX}
\sup_{N \in \N} \E\bigg[ \sup_{t \in [0,T]} \big| \widetilde{X}_N(t) \big|^{p-r+1} \bigg] < \infty.
\end{equation}
\end{Proposition} 

\begin{proof}
First of all, note that in the given setting the growth condition~\hyperlink{ass:pG}{\normalfont (pG$_{(r+1)/2}^{\sigma}$)} also holds true. Our main idea of proof is to show \eqref{eq:supEsupTildeX} by means of Gronwall's lemma.

As a first step, observe that for each $N \in \N$ the continuous-time tamed Euler scheme $\widetilde{X}_{N}$ satisfies
\begin{equation}
\label{eq:appA_last}
\E\bigg[ \sup_{t \in [0,T]} \big| \widetilde{X}_{N}(t) \big|^{p} \bigg] < \infty
\end{equation}
due to the taming of the drift and the diffusion coefficients in its construction, cf. Remark~3 in Sabanis~\cite{sabanis2016}. Note that one can not guarantee at the moment that this bound holds uniformly in~$N$; yet, the estimate \eqref{eq:appA_last} ensures the finiteness needed for applying Gronwall's lemma later on.

The next step is to establish the moment bound
\begin{equation}
\label{eq:appA_supsupEwidehatX}
\sup_{N \in \N} \sup_{t \in [0,T]} \E\Big[ \big| \widetilde{X}_{N}(t) \big|^p \Big] \le c.
\end{equation}
This is shown in a completely analogous manner to Lemma~2 in Sabanis~\cite{sabanis2016}, and we therefore omit a proof.
 
We now turn to estimates which allow to apply Gronwall's lemma in a final step. Put $\overline{p} := p-r+1 \in [2,p]$ as well as $\underline{t}_N := \lfloor tN/T \rfloor \cdot T/N$ for $t \in [0,T]$ and $N \in \N$. Fix $t \in [0,T]$ and $N \in \N$ for the moment.
First, applying Itô's formula to the Itô process~\eqref{eq:ItoStructure} and employing Assumption~\hyperlink{ass:K}{\normalfont (K$_{p}$)} yield
\begin{equation}
\label{eq:appA_hilf1}
\begin{split}
&\E\bigg[ \sup_{u \in [0,t]} \left( 1+ \big| \widetilde{X}_{N}(u) \big|^2 \right)^{\overline{p}/2} \bigg] \\
&\le \E\Big[ \big( 1 + |\xi|^2 \big)^{\overline{p}/2} \Big] + c \cdot \E\bigg[ \int_0^t \left( 1 + \big| \widetilde{X}_{N}(s) \big|^2 \right)^{(\overline{p}-2)/2} \cdot \left( 1 + \big| \widetilde{X}_{N}(\underline{s}_N) \big|^2 \right) \bigg] \, \mathrm{d}s \\
&\hspace{0.4cm}+ \E\bigg[ \int_0^t \left( 1 + \big| \widetilde{X}_{N}(s) \big|^2 \right)^{(\overline{p}-2)/2} \\
&\hspace{1.9cm}\cdot \bigg| \big( \widetilde{X}_{N}(s) - \widetilde{X}_{N}(\underline{s}_N) \big)^\top \cdot \frac{\mu\big( \underline{s}_N, \widetilde{X}_{N}(\underline{s}_N) \big)}{1+(T/N)^{1/2} \cdot \big| \widetilde{X}_{N}(\underline{s}_N) \big|^r} \bigg| \, \mathrm{d}s \bigg] \\
&\hspace{0.4cm}+ \overline{p} \cdot \E\bigg[ \sup_{u \in [0,t]} \int_0^u \left( 1 + \big| \widetilde{X}_{N}(s) \big|^2 \right)^{(\overline{p}-2)/2} \\
&\hspace{3.5cm}\cdot \widetilde{X}_{N}(s)^\top \cdot \frac{\sigma\big( \underline{s}_N, \widetilde{X}_{N}(\underline{s}_N) \big)}{1+(T/N)^{1/2} \cdot \big| \widetilde{X}_{N}(\underline{s}_N) \big|^r} \, \mathrm{d}W(s) \bigg].
\end{split}
\end{equation}
%Similarly to the proof of the preceding Proposition \ref{prop:EsupX}, we then conclude
%\[ \E\bigg[ \sup_{t \in [0,T]} \big| \widetilde{X}_{N}(t) \big|^{\overline{p}} \bigg] \le c \]
%where one particularly uses
% Hence, the desired inequality \eqref{eq:supEsupTildeX} follows immediately.
% To ensure finiteness of $\E[ \sup_{t \in [0,T]} | \widetilde{X}_{N}(t) |^{p-r}]$ for every $N \in \N$. whilst using \eqref{eq:EsupTildeX_CN} instead of stopping times 
% Moreover, it is easy to see that for each $N \in \N$ there exists a constant $C_N \in (0,\infty)$ such that
%\begin{equation}
%\label{eq:EsupTildeX_CN}
%\E\bigg[ \sup_{t \in [0,T]} \big| \widetilde{X}_{N}(t) \big|^{p} \bigg] \le C_N.
%\end{equation}
By the Young inequality, we obtain
\begin{equation}
\label{eq:appA_hilf2}
\begin{split}
&\E\bigg[ \int_0^t \Big( 1 + \big| \widetilde{X}_{N}(s) \big|^2 \Big)^{(\overline{p}-2)/2} \cdot \Big( 1 + \big| \widetilde{X}_{N}(\underline{s}_N) \big|^2 \Big) \, \mathrm{d}s \bigg] \\
%&\le \E\left[ \int_0^t \frac{\overline{p}-2}{\overline{p}} \cdot \Big( 1 + \big| \widetilde{X}_{N}(s) \big|^2 \Big)^{\overline{p}/2} + \frac{2}{\overline{p}} \cdot \Big( 1 + \big| \widetilde{X}_{N}(\underline{s}_N) \big|^2 \Big)^{\overline{p}/2} \, \mathrm{d}s \right] \\
%&\le \E\left[ \int_0^t \sup_{u \in [0,s]} \Big( 1 + \big| \widetilde{X}_{N}(u) \big|^2 \Big)^{\overline{p}/2} \, \mathrm{d}s \right] \\
&\le \int_0^t \E\bigg[ \sup_{u \in [0,s]} \Big( 1 + \big| \widetilde{X}_{N}(u) \big|^2 \Big)^{\overline{p}/2} \bigg] \, \mathrm{d}s. \\
\end{split}
\end{equation}
Moreover, the Cauchy--Schwarz inequality, the triangle inequality, Assumption~\hyperlink{ass:pG}{\normalfont (pG$_{r}^\mu$)}, the Young inequality, and \eqref{eq:appA_supsupEwidehatX} give
\begin{equation}
\label{eq:appA_hilf3}
\begin{split}
&\E\bigg[ \int_0^t \left( 1 + \big| \widetilde{X}_{N}(s) \big|^2 \right)^{(\overline{p}-2)/2} \cdot \bigg| \big( \widetilde{X}_{N}(s) - \widetilde{X}_{N}(\underline{s}_N) \big)^\top \cdot \frac{\mu\big( \underline{s}_N, \widetilde{X}_{N}(\underline{s}_N) \big)}{1+(T/N)^{1/2} \cdot \big| \widetilde{X}_{N}(\underline{s}_N) \big|^r} \bigg| \, \mathrm{d}s \bigg] \\
&\le c \cdot \sup_{u \in [0,T]} \E\bigg[ \Big( 1 + \big| \widetilde{X}_{N}(u) \big|^2 \Big)^{p/2} \bigg] \\
&\le c.
\end{split}
\end{equation}
Similarly to the derivations of \eqref{eq:hattrick} and \eqref{eq:hattrick2} in Lemma~\ref{prop:EsupX}, one utilizes the Burkholder--Davis--Gundy inequality, the Cauchy--Schwarz inequality, the growth condition~\hyperlink{ass:pG}{\normalfont (pG$_{(r+1)/2}^\sigma$)}, the inequality $\sqrt{x \cdot y} \le x/(2\rho) + y\rho/2$ for all $x,y \in [0,\infty)$ and $\rho \in (0,\infty)$, the Young inequality, and \eqref{eq:appA_supsupEwidehatX} to show
\begin{equation}
\label{eq:appA_hilf4}
\begin{split}
&\E\bigg[ \sup_{u \in [0,t]} \int_0^u \left( 1 + \big| \widetilde{X}_{N}(s) \big|^2 \right)^{(\overline{p}-2)/2} \cdot \widetilde{X}_{N}(s)^\top \cdot \frac{\sigma\big( \underline{s}_N, \widetilde{X}_{N}(\underline{s}_N) \big)}{1+(T/N)^{1/2} \cdot \big| \widetilde{X}_{N}(\underline{s}_N) \big|^r} \, \mathrm{d}W(s) \bigg] \\
&\le \frac{1}{2 \cdot \overline{p}} \cdot \E\bigg[ \sup_{u \in [0,t]} \Big( 1 + \big| \widetilde{X}_{N}(u) \big|^2 \Big)^{\overline{p}/2} \bigg] + c \cdot \sup_{u \in [0,T]} \E\bigg[ \Big( 1 + \big| \widetilde{X}_{N}(u) \big|^2 \Big)^{p/2} \bigg] \, \mathrm{d}s \\
&\le \frac{1}{2 \cdot \overline{p}} \cdot \E\bigg[ \sup_{u \in [0,t]} \Big( 1 + \big| \widetilde{X}_{N}(u) \big|^2 \Big)^{\overline{p}/2} \bigg] + c. \\
\end{split}
\end{equation}
Combining \eqref{eq:appA_hilf1}, Assumption~\hyperlink{ass:I}{\normalfont (I$_{p}$)}, \eqref{eq:appA_hilf2}, \eqref{eq:appA_hilf3}, and \eqref{eq:appA_hilf4} yields
\begin{equation*}
\begin{split}
\E\bigg[ \sup_{u \in [0,t]} \Big( 1 + \big| \widetilde{X}_{N}(u) \big|^2 \Big)^{\overline{p}/2} \bigg] &\le c + c \cdot \int_0^t \E\bigg[ \sup_{u \in [0,s]} \Big( 1 + \big| \widetilde{X}_{N}(u) \big|^2 \Big)^{\overline{p}/2} \bigg] \, \mathrm{d}s. \\
\end{split}
\end{equation*}
Applying Gronwall's lemma finally finishes the proof of this proposition.
\end{proof}

Lastly, the following proposition states that the continuous-time tamed Euler schemes converge strongly to the solution of the SDE \eqref{eq:SDE} with order $1/2$.

\begin{Proposition}
\label{prop:dist_X_tildeX}
Fix $q \in [1,\infty)$ and let the Assumptions \hyperlink{ass:I}{\normalfont (I$_{p}$)}, \hyperlink{ass:H}{\normalfont (H)}, \hyperlink{ass:K}{\normalfont (K$_{p}$)}, \hyperlink{ass:M}{\normalfont (M$_{a}$)}, and \hyperlink{ass:pL}{\normalfont (pL$_r^\mu$)} be satisfied for some $p, a \in [2,\infty)$ and $r \in [0,\infty)$ such that $p \ge 4r + 2$ and $q < \min\{ a, p/(2r+1) \}$. Moreover, let the functions $(\mu_N)_{N \in \N}$ and $(\sigma_N)_{N \in \N}$ be given by \eqref{eq:muN_sigmaN_tamedEuler}. Then there exists $C \in (0,\infty)$ such that for all $N \in \N$ it holds that
\begin{equation*}
%\label{eq:tildeX_closeTo_solution}
%\Big\Vert \big\Vert X - \widetilde{X}_N \big\Vert_\infty \Big\Vert_{L_q} \le C \cdot N^{-1/2}.
\bigg\Vert \sup_{t \in [0,T]} \big| X(t) - \widetilde{X}_N(t) \big| \bigg\Vert_{L_q} \le C \cdot N^{-1/2}.
\end{equation*}
%Let \hyperlink{ass:locL}{\normalfont (locL)}, \hyperlink{ass:H}{\normalfont (H)}, \hyperlink{ass:I}{\normalfont (I$_{p}$)}, \hyperlink{ass:K}{\normalfont (K$_{p}$)}, \hyperlink{ass:M}{\normalfont (M$_{a}$)}, and \hyperlink{ass:pL}{\normalfont (pL$_r^\mu$)} be satisfied for some $p, a \in [2,\infty)$ and $r \in [0,\infty)$ with $p \ge 4r + 2$. For each $N \in \N$ consider an arbitrary time discretization \eqref{eq:discretization} and let $\widetilde{X}_N$ be the corresponding continuous-time tamed Euler method. Then for all $q \in (0,\min\{ a,p/(2r+1) \})$ there exists $C \in (0,\infty)$ such that for all $N \in \N$ it holds
%\begin{equation}
%\label{eq:tildeX_closeTo_solution}
%\bigg\Vert \sup_{t \in [0,T]} \big| X(t) - \widetilde{X}_N(t) \big| \bigg\Vert_{L_q} \le C \cdot \max_{\ell \in \{ 0, \ldots, N-1 \}} (t_{\ell+1}-t_\ell)^{1/2}.
%\end{equation}
\end{Proposition}

\begin{proof}
Essentially, the proof of Theorem 3 in Sabanis~\cite{sabanis2016} carries over here and is therefore omitted.
%To this end, one utilizes in turn that
%\begin{equation*}
%\sup_{N \in \N} \sup_{t \in [0,T]} \E\left[ \big| \widetilde{X}_N(t) \big|^{p} \right] < \infty
%\end{equation*}
%and that
%\[ \sup_{t \in [0,T]} \Big\Vert \big| X(t) - \widetilde{X}_N(t) \big| \Big\Vert_{L_q} \le c \cdot N^{-1/2} \]
%holds for all $N \in \N$.
\end{proof}

%\newpage
%%%%%%%%%%%%%%%%%%%%%%%%%%%%%%%%%%%%%%%%%%%%%%%%%%%%%%
%\bibliographystyle{splncs03}
\bibliographystyle{amsplain}
%\bibliography{literatur}
\bibliography{\jobname}
%%%%%%%%%%%%%%%%%%%%%%%%%%%%%%%%%%%%%%%%%%%%%%%%%%%%%%

\end{document}